\documentclass[11pt,a4paper]{article}

\usepackage{tikz-cd}

\usepackage{amsfonts}
\usepackage{amsmath}
\usepackage{amssymb}
\usepackage{amsthm}

\newtheorem{theorem}{Theorem}
\newtheorem{proposition}{Proposition}
\newtheorem{lemma}{Lemma}
\newtheorem*{corollary}{Corollary}
\newtheorem*{theorem*}{Theorem}
\theoremstyle{definition}
\newtheorem{definition}{Definition}
\newtheorem*{remark}{Remark}
\newtheorem*{example}{Example}

\begin{document}


\title{\bf Hypergroups over the group and generalizations of Schreier's theorem on group extensions}

\author{\textsc{S.H.Dalalyan}\\ {\rm\textsl{Yerevan State University}}}

\date{\today}

\maketitle

\begin{abstract}
Let $H$ be a group, $m$ be a positive integer, $Ext_m H$ be
the set of all isomorphic in $G$ classes of group monomorphisms
$\varphi: H \rightarrow G$ such that index of $\varphi(H)$ in $G$
is $m$.
The main goal of this paper is to describe  the elements of $Ext_m H$
in terms of a new concept of hypergroups over the group.
The obtained result is a very broad generalization of the Schreier
theorem (1926).
As an application, a series of intermediate generalizations is obtained;
particularly, a description of the set $Ext (H, Q)$ of isomorphic classes
of all extensions of a noncommutative group $H$ by $Q$.
\end{abstract}




\section{Introduction }

In group theory, there are two diametrically opposite understandings
of the expression
``the extension of one group by another group''.
Let
$$
E \rightarrow H \rightarrow G \rightarrow Q \rightarrow E
$$
be a short exact sequence of groups.
Then $G$ is called an extension of $Q$ by $H$,
as well as an extension of $H$ by $Q$.
The second name is in better agreement
with the tradition established in other areas of mathematics
(fields extension, rings extension, \textit{etc.}).
In this paper, more generally, we shall call {\it an extension
of a group $H$} any monomorphism of groups
$\varphi: H \rightarrow G$.

  For any extension $\varphi: H \rightarrow G$,
  the image $\varphi (H)$ is a subgroup of $G$,
 isomorphic to $H$, and actually is identified with $H$.
  The index of the subgroup $\varphi(H)$ of the group $G$
  is called the {\it degree} of the extension $\varphi$.

 Let $Ext_m H$ be the set of classes of isomorphic in $G$
 extensions of the group $H$ of a given degree $m$.
 The main goal of this paper is to introduce a concept
 of a hypergroup over the group  and to describe in terms
 of this concept the elements of $Ext_m H$.

The structure of a hypergroup over the group arises
in a standard way on each complementary (sub)set $M$
to a subgroup $H$ of any group $G$
and the corresponding object is denoted by $M_H$.
The concept of a hypergroup $M_H$ over the group
unifies and generalizes the concepts
of a group, field, linear space over the field.
All hypergroups over the group together with
their morphisms form a category (Sec. 2).

The above mentioned standard construction of a
hypergroup over the group is universal:
any hypergroup $M_H$ over the group is isomorphic
to a hypergroup over the group, which is obtained
in such a standard way.
A key role in the proof of this property of universality
plays the notion of an (outer) exact product
associated with a hypergroup over the group.
This notion generalizes the notions of the direct product,
the semidirect product and the general product in the sense
of B. Neumann  (Sec. 3).

Let $Hg (H, M)$ be the set of classes of isomorphic hypergroups $M_H$
over the group $H$.
In Section 4, we define a group $\mathcal G$ and an action $\mathcal A$
of this group on $Hg (H, M)$.

Let $Hg (H, M) / {\mathcal A}$ be the corresponding quotient-set
(i.e. the set of orbits under this action).
The main result of this paper is the following theorem of Sec. 5.

\begin{theorem}
\label{th1}
 There is a canonical bijection between the  sets $Ext_m (H)$ and
 $Hg (H, M) / {\mathcal A}$, where $m = |M|$.
\end{theorem}

This theorem is a very broad generalization of the Schreier theorem (1926),
which can be obtained from this theorem by imposing  a number of restrictions.

Any hypergroup $M_H$ over the group has a system of structural mappings
$\Omega = (\Phi, \Psi, \Xi, \Lambda)$.
Firstly, by considering the set $Hg_0(H, M)$ of all hypergroups $M_H$ over the group
with a trivial action of the group $H$ (i.e. with a trivial structural
mapping $\Phi$),
we prove that the set $Ext_{m,0} (H)$, consisting of all normal extensions $\varphi$
(it means that $\varphi(H)$ is a normal subgroup of $G$),
is bijective to $Hg_0(H, M) \slash {\mathcal A}$
(Theorem $\ref{th11}$, Section 6).

Then, we consider the subsets $Hg_0(H, M, [\Xi])$ and $Hg_0(H, M, \Xi)$ of $Hg_0(H, M)$,
consisting of all hypergroups over the group with a fixed isomorphity class $[\Xi]$,
in the first case, and with a fixed $\Xi$, in the second case,
and prove that the sets of orbits $Hg_0(H, M, [\Xi]) \slash {\mathcal A}$ and
$Hg_0(H, M, \Xi) \slash {\mathcal A}_0$ are bijective to the set $Ext (H, Q)$
of isomorphic classes of all extensions of $H$ by the group $Q = (M, \Xi)$
(Theorems $\ref{th12}$ and $\ref{th13}$, Section 6).

Finally, by considering only the commutative groups $H$, we get the Schreier's theorem
(Theorem $\ref{th16}$, Sec. 7).

We want to emphasize that Theorems $\ref{th12}$ and $\ref{th13}$
are generalizations of Schreier's theorem in the most general case
of a noncommutative group $H$.


\section{The standard construction \\
of a hypergroup over the group}

\label{sec.2}

Let $G$ be a multiplicative group and $H$ be a subgroup of $G$. Let
$$
Q = H \backslash G = \{ H \cdot a, \, a \in G  \}
$$
be the corresponding right quotient-set and
$$
\psi: G \rightarrow Q, \quad a \mapsto H \cdot a
$$
be the associated surjective right quotient-map.
Recall that a {\it section} of a surjective map $\psi$ is a map $\sigma$,
satisfying the condition $\sigma \circ \psi = id$.
A {\it right transversal} to a subgroup $H$ of a group $G$
is a subset $M$ of $G$ such that $|M \cap (H \cdot a)| = 1$ for any element $a \in G$.
A subset $M$ of $G$ is a right transversal to the subgroup $H$ if and only if
$M$ is the image of a section $\sigma$ of the right quotient-map $\psi$.

\begin{definition}
It is said that a (multiplicative) group $G$ is an {\it exact product}
of its subsets $A$ and $B$ if for any $x \in G$ there are unique $a \in A$ and
$b \in B$ such that $x = a \cdot b$. In this case we write $G = A \odot B$.
A subset $M\subset G$ is called a {\it right complementary set} to a subgroup $H$
of a group $G$ if $G = H \odot M$.
It is not difficult to check that $M$ is a right complementary set to
a subgroup $H$ of a group $G$ if and only if $M$ is a right transversal to $H$.
\end{definition}

Let $G = H\odot M$. Define the mappings
\begin{align*}
&\Phi: M \times H \rightarrow M,    & & \Phi(a, \alpha) = a^\alpha,  &&
&&\Psi: M \times H \rightarrow H,    & & \Psi(a, \alpha) = {^a}\alpha,\\
&\Xi: M \times M \rightarrow M,     & & \Xi(a, b) =  [a, b], &&
&&\Lambda: M \times M \rightarrow H, & & \Lambda(a, b) = (a, b),
\end{align*}
by formulas
\begin{align*}
&&&
&& \tag{F1}\label{F1}
a \cdot \alpha = {^a}\alpha \cdot a^\alpha, \qquad
a \cdot b = (a, b) \cdot [a, b] &&
\end{align*}
and let $\Omega = (\Phi, \Psi, \Xi, \Lambda)$.

\begin{theorem}
\label{th2}
The system of mappings $\Omega$ has the following properties.
\begin{itemize}
\item[\bf{ P1)}]
The mapping $\Xi$ is a binary operation such that
\begin{itemize}
\item[$(i)$]
$(M, \Xi)$ is a right quasigroup, i. e.
any equation $[x, a] = b$ with elements $a, b \in M$
has a unique solution in $M$;
\item[$(ii)$]
the binary operation $\Xi$ has a left neutral element $o$,
which means that  $[o, a] = a$ for any $a \in M$.
\end{itemize}

\item[\bf{P2)}]  The mapping $\Phi$ is an action of the group $H$ on the set $M$, that is
\begin{itemize}
\item[$(i)$]
$(a^\alpha)^\beta = a^{\alpha \cdot \beta}$
for any pair of elements $\alpha, \beta \in H$ and for every $a \in M$;
\item[$(ii)$]
$a^\varepsilon = a$ for any  $a \in M$, where $\varepsilon$ is the neutral element of the group $H$.
\end{itemize}

\item[\bf{P3)}]  The mapping $\Psi$ is surjective.

\item[\bf{P4)}]  The following identities A1) - A5) hold:
\begin{itemize}
\item[$(A1)$] $\quad {^a}(\alpha \cdot \beta)  = {^a}\alpha \cdot {^{a^\alpha}}\beta;$

\item[$(A2)$] $\quad [a, b]^\alpha  = [a^{^b\alpha}, b^\alpha];$

\item[$(A3)$] $\quad (a, b) \cdot \, {^{[a, b]}}\alpha  =
{^a}({^b}\alpha) \cdot (a^{{^b}\alpha}, b^\alpha);$

\item[$(A4)$] $\quad [[a, b], c]) = [a^{(b, c)}, [b, c]];$

\item[$(A5)$] $\quad (a, b) \cdot ([a, b], c) = {^a}(b, c) \cdot (a^{(b, c)}, [b, c]).$
\end{itemize}
\end{itemize}
\end{theorem}

\begin{proof}
The relationships $(a^\alpha)^\beta = a^{\alpha \cdot \beta}$ and (A1)
are obtained by applying (F1)  to the identity of associativity
$$
(a \cdot \alpha) \cdot \beta = a \cdot (\alpha \cdot \beta), \quad
a \in M, \, \alpha, \beta \in H,
$$
and by using the uniqueness of the decomposition, associated with the exact product
$G = H\odot M$.

Similarly, relations (A2), (A3) and (A4), (A5) can be obtained, respectively,
from the identities of associtivity
$$
(a \cdot b) \cdot \alpha = a \cdot (b \cdot \alpha), \quad
(a \cdot b) \cdot c = a \cdot (b \cdot c)
$$
with $a, b, c \in M, \, \alpha \in H$,
whereas the relation $a^\varepsilon = a$ (together with the relation ${^a}\varepsilon = \varepsilon$)
can be obtained from the equality
$$
a \cdot \varepsilon = \varepsilon \cdot a.
$$

Let
$$
e = \theta \cdot o, \quad \theta \in H, o \in M
$$
be the decomposition of the neutral element $e = \varepsilon \in H \odot M$ of $G$.
Then for any $a \in M$ we will have
$$
\varepsilon \cdot a = \theta \cdot o \cdot a = (\theta \cdot (o, \, a)) \cdot[o, \, a].
$$
Consequently, $[o, a] = a$ (and $\theta \cdot (o, a) = \varepsilon$).

Since
$$
(\alpha \cdot \theta) \cdot o = \alpha = \theta \cdot (o \cdot \alpha) =
(\theta \cdot {^o}\alpha) \cdot o^\alpha,
$$
we get the relations ${^o}\alpha = \theta^{-1} \cdot \alpha \cdot \theta$
and $o^\alpha = o$.
As a consequence of the first relation, the mapping $\Psi$ is surjective.

To complete the proof of Theorem $\ref{th2}$
it remains to prove property P1(i).

Suppose that for any $a \in M$ the elements $a^{(-1)} \in H, \, a^{[-1]} \in M$ are
the uniquely determined elements, satisfying
$$
a^{-1} = a^{(-1)} \cdot a^{[-1]}.
$$
Then
$$
a^{(-1)} \cdot (a^{[-1]}, \, a) \cdot [a^{[-1]}, \, a] = a^{(-1)} \cdot a^{[-1]} \cdot a
= a^{-1} \cdot a = e = \theta \cdot o,
$$
and consequently
$$
[a^{[-1]}, \, a] = o, \quad a^{(-1)} \cdot (a^{[-1]}, \, a)) = \theta.
$$
The first equalities means that $a^{[-1]}$ is a left inverse element to $a$
with respect to the operation $\Xi$, and the second equality can be considered as a
variant of the definition of $a^{(-1)}$.

\begin{lemma}
\label{lm0}
For any elements $x, a, b \in M$ we have
$$
[x, a] = b \quad \Leftrightarrow \quad
x = [b^{a^{(-1)}}, a^{[-1]}], \quad
(x, a) \cdot {^b}(a^{(-1)}) \cdot (b^{a^{(-1)}}, a^{[-1]})  = \varepsilon.
$$
\end{lemma}

\begin{proof}
By  formula (F1)
we have the following chain of equivalent equalities:
$$
[x, a] = b \quad \Leftrightarrow \quad
(x, a) \cdot [x, a] = (x, a) \cdot b \quad \Leftrightarrow \quad
x \cdot a = (x, a) \cdot b \quad \Leftrightarrow \quad
$$
$$
x = (x, a) \cdot b \cdot a^{-1} \quad \Leftrightarrow \quad
x = (x, a) \cdot b \cdot a^{(-1)} \cdot a^{[-1]} \quad \Leftrightarrow \quad
$$
$$
x = (x, a) \cdot {^b}(a^{(-1)}) \cdot b^{a^{(-1)}} \cdot a^{[-1]} \quad \Leftrightarrow \quad
$$
$$
x = (x, a) \cdot {^b}(a^{(-1)}) \cdot (b^{a^{(-1)}}, a^{[-1]}) \cdot [b^{a^{(-1)}}, a^{[-1]}] \quad \Leftrightarrow \quad
$$
$$
x = [b^{a^{(-1)}}, a^{[-1]}], \quad  (x, a) \cdot {^b}(a^{(-1)}) \cdot (b^{a^{(-1)}}, a^{[-1]}) = \varepsilon.
$$
\end{proof}

\begin{corollary}
The equation $[x, a] = b$ for $a, b \in M$ can have
no more than one solution in $M$.
Such a solution can be only $x = [b^{a^{(-1)}}, a^{[-1]}]$.
This value of $x$ will be a solution if and only if the equality
$$
\hspace{30mm}
([b^{a^{(-1)}}, a^{[-1]}], a) \cdot {^b}(a^{(-1)}) \cdot (b^{a^{(-1)}}, a^{[-1]})
 = \varepsilon
 \hspace{15mm} (S)
$$
holds.
\end{corollary}

Substituting the value $x = [b^{a^{(-1)}}, a^{[-1]}]$
in equation, we check that it is a solution:
$$
[[b^{a^{(-1)}}, a^{[-1]}], a] = [[b^{a^{(-1)} \cdot (a^{[-1]}, a)}, [a^{[-1]}, a]] =
[b^\theta, o] = (b^\theta)^{\theta^{-1}} = b.
$$
Here we use the relation $[a, o] = a^{\theta^{-1}}$,
which can be deduced
(together with the relation $(a, o) = {^a}(\theta^{-1})$)
 from the equalities
$$
(a, o) \cdot [a, o] = a \cdot o = a \cdot \theta^{-1} = {^a}(\theta^{-1}) \cdot a^{\theta^{-1}}.
$$
Thus, Theorem \ref{th2} is completely proved.
Note that along the proof above we have also proved the identity $(S)$.
\end{proof}

\begin{definition}
\label{df2}
Let $H$ be an arbitrary group.
A {\it right hypergroup over the group $H$} is
a set $M$ together with a system of structural mappings
$\Omega = (\Phi, \Psi, \Xi, \Lambda)$, satisfying conditions
P1, P2, P3, P4. We denote by $M_H$ this hypergroup over the group.
\end{definition}

The first definition of a hypergroup over the group was given in $\cite{D1}$.
That definition was very cumbersome and contained some extra conditions
(which can be deduced from a smaller number of conditions).
An improved definition, close to Definition 2, is given in $\cite{D2}$.

\begin{remark} Similarly, by considering the left quotient-set
$G / H = \{ a \cdot H, \quad a \in G \}$
one can define the left hypergroup over the group $H$.
All notions and results for right hypergroups over the group
allow dual notions and results in the case of
left hypergroups over the group.
In the rest of this paper we consider only right hypergroups over
the group and omit the word "right".
\end{remark}

 Based on Definition  $\ref{df2}$ and Theorem $\ref{th2}$,
 the following theorem can be formulated.

 \begin{theorem}
\label{th3}
A hypergroup $M_H$ over the group can be canonically associated to any subgroup $H$ of
an arbitrary group $G$ and each complementary set $M$ to $H$ in $G$.
\end{theorem}

We call the corresponding construction the
{\it standard construction} of a hypergroup over the group.

\begin{remark}
Using the standard construction of a hypergroup over the group,
we have an opportunity for any subgroup $H$ of an arbitrary group $G$
to define an (induced by binary operation of $G$)
operation on the quotient-set $Q$ of $G$ by $H$:
it suffices
to consider an arbitrary section $\sigma$ of the
quotient-map $\psi: G \rightarrow Q$,
to take a  transversal
$M = \sigma(Q)$ to the subgroup $H$, to construct the corresponding
hypergroup over the group $M_H$ and
to transfer the binary operation
$\Xi $ of $M$ on $Q$ by $\sigma$.

Thus, this operation on the quotient-set $Q$
depends on the choice of the section $\sigma$.
Therefore, in general, it is not determined uniquely
(even up to an isomorphism) and it is not a group operation
(it determines on $Q$ a structure of a right quasigroup
with a left neutral element).
If $H$ is a normal subgroup of $G$,
for any section $\sigma$
the induced binary operations on $Q$ are
isomorphic to a unique (quotient-)group operation.
\end{remark}

\begin{definition}
Let $M_H$ and $M'_{H'}$ be two hypergroups over the group
with systems of structural mappings
$\Omega = (\Phi, \Psi, \Xi, \Lambda)$ and
$\Omega' = (\Phi', \Psi, \Xi', \Lambda')$, respectively.

A  {\it morphism of hypergroups over the group}
$$
f: M_H \rightarrow M'_{H'}
$$
is a  pair $f = (f_0, \, f_1)$, consisting of a homomorphism of groups
$$
f_0: H \rightarrow H'
$$
and of a map of sets
$$
f_1: M \rightarrow M',
$$
{\it preserving the structural mappings},
i.e., such that the following diagrams commute:

\vspace{5mm}

\begin{minipage}{0.5\textwidth}
   \begin{tikzcd}
   M \times H \arrow{r}{\Phi} \arrow[swap]{d}{f_1 \times f_0} &
   M \arrow{d}{f_1} \\
   M' \times H'  \arrow{r}{\Phi'} & M'
   \end{tikzcd}
\end{minipage}
\hspace{0.5cm}
\begin{minipage}{0.5\textwidth}
   \begin{tikzcd}
   M \times H \arrow{r}{\Psi} \arrow[swap]{d}{f_1 \times f_0} &
   H \arrow{d}{f_0} \\
   M' \times H'  \arrow{r}{\Psi'} & H'
   \end{tikzcd}
\end{minipage}

\vspace{5mm}

\begin{minipage}{0.5\textwidth}
   \begin{tikzcd}
   M \times M \arrow{r}{\Xi} \arrow[swap]{d}{f_1 \times f_1} &
   M \arrow{d}{f_1} \\
   M' \times M'  \arrow{r}{\Xi'} & M'
   \end{tikzcd}
\end{minipage}
\hspace{0.5cm}
\begin{minipage}{0.5\textwidth}
   \begin{tikzcd}
   M \times M \arrow{r}{\Lambda} \arrow[swap]{d}{f_1 \times f_1} &
   H \arrow{d}{f_0} \\
   M' \times M'  \arrow{r}{\Lambda'} & H'
   \end{tikzcd}
\end{minipage}
\end{definition}

The composition of two morphisms
$$
f: M_H \rightarrow M'_{H'}, \quad
g: M'_{H'} \rightarrow M''_{H''}
$$
is defined componentwise:
$f \circ g = (f_0 \circ g_0, \, f_1 \circ g_1)$.

The identical morphism $1_{M_H}$ of a hypergroup $M_H$ over the group
is given by the pair $(1_H, \, 1_M)$.

\begin{proposition}
\label{pr1}
All hypergroups over the group and their morphisms
together with composites of morphisms and identical morphisms
form a category.
\end{proposition}

\begin{proof}
The verification of axioms of category is trivial.
\end{proof}

As a consequence of this proposition, we can apply to the category of
hypergroups over the group all notions and results of the general
theory of categories.
Emphasize that an isomorphism of hypergroups over the group consists of
an isomorphism of groups and a bijection of sets.

We denote the category of hypergroups over the group by $Hg$.
All hypergroups over the fixed group $H$ and their morphism with identical
first component form a subcategory $Hg \slash H$ of the category $Hg$.
If the group $H = \{ \varepsilon \}$ is trivial,
we denote $Hg\slash H$ by $Hg_0$.

\begin{example}
Let $H = \{ \varepsilon \}$ be a trivial group. Then the structural mappings
$\Phi, \Psi, \Lambda$ are uniquely defined and $\Xi$ is an associative binary
operation on $M$ with a left neutral element $o$ and with a left inverse
element $a^{[-1]}$ for any element $a$ from $M$. Consequently, $(M, \Xi)$
is a group.
Therefore the correspondence $M_H \mapsto (M, \Xi)$
determines an isomorphism between $Hg_0$ and the group category.
\end{example}

\begin{example}
Let $k$ be a field (generally with a noncommutative multiplication).
Suppose that $L$ is a linear space over $k$ with
a right multiplication by elements of $k$.
Let $(M, \Xi)$ be the additive group of $L$ and
$H$ be the multiplicative group $k^*$ of $k$.
Let $\Phi$ be the multiplication of elements $a \in M = L$
by elements $\alpha \in H \subset k$ (that is $a^\alpha = a \cdot \alpha$),
and mappings $\Psi$ and $\Lambda$ be trivial, that is
$^a\alpha = a$ and $(a, b) = \varepsilon$
for any $a, b \in M$, $\alpha \in H$.
Then properties $P1 - P4$ hold and, consequently, $M_H$ is
a hypergroup over the group.
Thus we have a canonical injection of the category of linear spaces
over $k$ in $Hg\slash k^*$.
 \end{example}

 \begin{example}
  Suppose that $k$ is an arbitrary field (generally, with a
 noncommutative multiplication).
 Let $(M, \Xi)$ be the additive group,
 $H = k^*$ be the multiplicative group of $k$,
 $\Phi (a, \alpha) = a \cdot \alpha$,
 $\Psi (a, \alpha) = \alpha$,
 $\Lambda(a, b) = \varepsilon$
 for all $\alpha \in H, \, a, b \in M$.
 Then the system of structural mappings
 $\Omega = (\Phi, \Psi, \Xi, \Lambda)$
 determines a hypergroup $M_H$.
 By assigning $k \mapsto M_H$ we get an injective functor
 from the category of (noncommutative) fields to $Hg$.
 \end{example}


\section{The universality property \\
of standard construction}

\label{sec.3}

In this section we prove a fundamental property of universality
for the standard construction of a hypergroup over the group.
In Theorem $\ref{th5}$, we prove that any hypergroup over the group,
up to isomorphism, may be obtained by the standard construction.
This result implies the following principle:

{\it Any property, which is true for all hypergroups over the group
that are obtained by the standard construction, is true for an
arbitrary hypergroup over the group}.

The proof of the property of universality is based on properties that are
postulated in the Definition $\ref{df2}$ of a hypergroup over
the group, but also uses some additional properties of hypergroups
over the group, which are simple consequences of these postulated conditions.
Some these properties
we have already got in the proof of Theorem $\ref{th2}$,
but only for the hypergroups over the group
arising by the standard construction.
To have a possibility to apply these properties in the general case
(to hypergroups over the group defined by Definition $\ref{df2}$)
we have to deduce these properties directly from the properties
postulated in Definition $\ref{df2}$.
Thus, the above principles will be comprehensive.

\begin{proposition}
\label{pr2}
Let  $M_H$ be an arbitrary hypergroup over the group,
$\varepsilon$ be the neutral element of the group $H$,
$o$ be the left neutral element of the operation $\Xi$
and $\theta = (o, o)^{-1}$.
Then  for any $a \in M, \, \alpha \in H$ the following equalities hold:
\begin{itemize}
\item[]
\begin{itemize}
\item[$(A6)$] $\quad ^a\varepsilon = \varepsilon$;
\item[$(A7)$] $\quad o^\alpha = o$;
\item[$(A8)$] $\quad ^o\alpha = \theta^{-1} \cdot \alpha \cdot \alpha$;
\item[$(A9)$] $\quad  (o, a) = \theta^{-1}$;
\item[$(A10)$] $\quad [a, o] = a^{\theta^{-1}}$;
\item[$(A11)$] $\quad (a, o) = {^a}(\theta^{-1})$.
\end{itemize}
\end{itemize}

\begin{proof}
(A6). It is sufficient to substitute $\alpha = \varepsilon$ in (A1) and to use P2(ii).\\
(A7). Substitute $a = o$ in (A2), use P1(ii), P1(i) and P3. \\
(A8). Substitute $a = b = o$ in (A3), use (A7), P3. \\
(A9). Substitute $a = b = o$ in (A5), use (A8), (A7), P1(ii). \\
(A10). Substitute $b = o$ in (A4), use P1(i). \\
(A11). Substitute $b = c = o$ in (A5) , use (A10).
\end{proof}
\end{proposition}

Let $M_H$ be an arbitrary hypergroup over the group.
Denote by $\overline G$ the Cartesian product $H \times M$ of
sets $H$ and $M$. Here, it is convenient to write
the elements of $\overline G$ in a form
$\alpha a, \, \alpha \in H, a \in M$.

We  define a binary operation on $\overline G$  by the formula
\begin{align}\tag{F2}\label{F2}
\alpha a \cdot \beta b = (\alpha \cdot {^a}\beta \cdot (a^\beta, b)) [a^\beta, b].
\end{align}

\begin{lemma}
\label{lm2}
The operation $(F2)$ is associative:
$$
(\alpha a \cdot \beta b) \cdot \gamma c = \alpha a \cdot (\beta b \cdot \gamma c).
$$
\end{lemma}

\begin{proof}
According to (F2) we have
$$
(\alpha a \cdot \beta b) \cdot \gamma c =
((\alpha \cdot {^a}\beta \cdot (a^\beta, b))\cdot {^{[a^\beta, b]}}\gamma \cdot ([a^\beta, b]^\gamma, \, c))
[[a^\beta, b]^\gamma, \, c]
$$
and
$$
\alpha a \cdot (\beta b \cdot \gamma c) =
(\alpha \cdot {^a}(\beta \cdot {^b}\gamma \cdot (b^\gamma, c)) \cdot
(a^{\beta \cdot {^b}\gamma \cdot (b^\gamma, c)}, \, [b^\gamma, c])
[a^{\beta \cdot {^b}\gamma \cdot (b^\gamma, c)}, \, [b^\gamma, c]].
$$
Here
$$
(\alpha \cdot {^a}(\beta \cdot {^b}\gamma \cdot (b^\gamma, c)) \cdot (a^{\beta \cdot {^b}\gamma \cdot (b^\gamma, c)}, \, [b^\gamma, c]) =
$$
$$
\alpha \cdot {^a}\beta \cdot {^{a^\beta}}({^b}\gamma) \cdot {^{a^{\beta \cdot {^b}\gamma}}}(b^\gamma, c) \cdot
(a^{\beta \cdot {^b}\gamma \cdot (b^\gamma, c)}, \, [b^\gamma, c]) =
$$
$$
\alpha \cdot {^a}\beta \cdot (a^\beta, \, b) \cdot {^{[a^\beta, \, b]}}\gamma
\cdot ((a^\beta)^{{^b}\gamma}, \, b^\gamma)^{-1}
 \cdot {^{a^{\beta \cdot {^b}\gamma}}}(b^\gamma, c) \cdot
(a^{\beta \cdot {^b}\gamma \cdot (b^\gamma, c)}, \, [b^\gamma, c]) =
$$
$$
\alpha \cdot {^a}\beta \cdot (a^\beta, \, b) \cdot {^{[a^\beta, \, b]}}\gamma
\cdot ([a^{\beta \cdot {^b}\gamma}, \, b^\gamma], \, c) =
$$
$$
((\alpha \cdot {^a}\beta \cdot (a^\beta, b))\cdot {^{[a^\beta, b]}}\gamma \cdot ([a^\beta, b]^\gamma, \, c))
$$
according to (A1), (A3), (A5), (A2), and
$$
[[a^\beta, b]^\gamma, \, c] =
[a^{\beta \cdot {^b}\gamma \cdot (b^\gamma, c)}, \, [b^\gamma, c]].
$$
according to (A2), (A4).
\end{proof}

\begin{lemma}
\label{lm3}
 The binary operation $(F2)$ has a left neutral  element $\theta o$.
 \end{lemma}

 \begin{proof}
  For any element $\alpha a \in {\overline G}$ we have
 $$
 \theta o \cdot \alpha a = (\theta \cdot {^o}\alpha \cdot (o^\alpha, \, a))[o^\alpha, \, a] = \alpha a.
 $$
 according to (F2), (A8), (A7), (A9), P1(ii).
 \end{proof}

 We denote the unique solution of an equation $[x, a] = o$ by $a^{[-1]}$
 and call a {\it left inverse element to $a$ in the quasigroup $(M, \Xi)$}.

\begin{lemma}
 \label{lm4}
 Any element $\alpha a$ of $\overline G$ has a unique left inverse element
 $\xi x \in {\overline G}$, defined by the relation $\xi x \cdot \alpha a = \theta o$.
 Namely,
 $$
 \hspace{20mm}
 x = (a^{[-1]})^{\alpha^{-1}}, \quad
 \xi = \theta \cdot (a^{[-1]}, \, a)^{-1} \cdot ({^x}\alpha)^{-1}.
 \hspace{10mm}  (F3)
 $$
 \end{lemma}

 \begin{proof}
 By (F2), $\xi x \cdot \alpha a = \theta o$ if and only if
 $$
 \xi \cdot {^x}\alpha \cdot (x^\alpha, \, a) = \theta, \quad
 [x^\alpha, \, a] = o.
 $$
 This system has a unique solution $(F3)$.
  \end{proof}

Combining the results of Lemmas $\ref{lm2}$, $\ref{lm3}$ and $\ref{lm4}$,
we obtain the following theorem.

\begin{theorem}
\label{th4}
For any hypergroup $M_H$ over the group,
the set ${\overline G} = H \times M$ together with the binary operation
defined in $(F2)$ is a group.
\end{theorem}

\begin{proof}
It suffices to use that every associative right quasigroup with a left neutral element is a group.
\end{proof}

\begin{definition}
Let $M_H$ be an arbitrary hypergroup over the group.
The group with a basic set $H \times M$ and a multiplication operation $(F2)$
is called an {\it (outer) exact product, associated with the hypergroup $M_H$}
and is  denoted by $H \odot M$.
\end{definition}

\begin{example}
Let $M_H$ be a hypergroup over the group with trivial
structural mappings $\Phi, \Psi, \Lambda$:
$$
a^\alpha = a, \quad {^a}\alpha = \alpha, \quad  (a, b) = \varepsilon
$$
for any $\alpha \in H, \, a, b \in M$.
Then $(M, \, \Xi)$ is a group and the outer exact product associated with the
hypergroup $M_H$ coincides with the direct product of groups $H$ and $(M, \, \Xi)$.
\end{example}

\begin{example}
Suppose that only $\Phi$ and $\Lambda$ are trivial:
$a^\alpha = a, \, (a, b) = \varepsilon$ for a hypergroup $M_H$.
Then again $(M, \, \Xi)$ is a group and the outer exact product associated with the
hypergroup $M_H$ is a semidirect product of $H$ by $(M, \, \Xi)$
(see for example $\cite{R}$, p. 167).
\end{example}

\begin{example}
If we have only one condition $(a, \, b) = \varepsilon$ for $M_H$,
then  $(M, \, \Xi)$ is a group as well, and the corresponding exact product
is a general product of $H$ and $(M,\, \Xi)$ in the sense of B. Neumann
$\cite{N}$ (see also $\cite{K}$, p. 485).
\end{example}

Thus, the notion of an (outer) exact product, associated with a hypergroup over
the group, generalizes the notions of a direct product, of a semidirect product,
of a general product of groups.

Using the construction of the outer exact product, associated with a hypergroup over the group,
we show that the standard construction of a hypergroup over the group is universal
in the following sense.

\begin{theorem}
\label{th5}
Any hypergroup over the group is isomorphic to a hypergroup over the group,
arising by the standard construction.
More precisely, let $M_H$ be an arbitrary hypergroup over the group and
${\overline G} = H \odot M$.
Consider the maps
$$
f_0: H \rightarrow {\overline G}, \quad \alpha \mapsto {\overline \alpha} = (\alpha \cdot \theta) o,
\quad \quad
f_1: M \rightarrow {\overline G}, \quad a \mapsto {\overline a} = \varepsilon a,
$$
and their images  ${\overline H} = {\rm im} \, f_0$, ${\overline M} = {\rm im} \, f_1$.
Then $\overline H$ is a subgroup of $\overline G$, $\overline M$ is a complementary set
to $\overline H$ and we have a hypergroup over the group ${\overline M}_{\overline H}$
with a system of structural mappings
 ${\overline \Omega} = ({\overline \Phi}, {\overline \Psi}, {\overline \Xi}, {\overline \Lambda})$.
This hypergroup over the group is isomorphic to $M_H$.
Such an isomorphism is established by a pair
${\overline f} = ({\overline f}_0, {\overline f}_1)$, consisting of corestrictions of $f_0$ and $f_1$:
$$
{\overline f}_0: H \rightarrow {\overline H}, \quad {\overline f}_0 (\alpha) = f_0 (\alpha), \quad
{\overline f}_1: M \rightarrow {\overline M}, \quad {\overline f}_1 (a) = f_1 (a).
$$
\end{theorem}

\begin{proof}
It is evident that $f_0$ and $f_1$ are injective maps and $f_0$ is a  homomorphism, since
$$
(\alpha \cdot \theta) o \cdot (\beta \cdot \theta) o =
(\alpha \cdot \theta \cdot {^o}(\beta \cdot \theta) \cdot (o^{\beta \cdot \theta}, \, o)) [o^{\beta \cdot \theta}, \, o] =
((\alpha \cdot \beta) \cdot \theta) o
$$
according to (F2), (A8), (A7), P1(ii).
Thus, ${\overline f}_0$ is an isomorphism of groups and ${\overline f}_1$ is a bijection of sets.

\begin{lemma}
\label{lm5}
For any elements $\xi \in H, \, x \in M$
$$
{\overline \xi} \cdot {\overline x} = (\xi \cdot \theta) o \cdot \varepsilon x =  \xi x.
$$
\end{lemma}

\begin{proof}
We have
$$
(\xi \cdot \theta) o \cdot \varepsilon x =
(\xi \cdot \theta \cdot {^o}\varepsilon \cdot (o^\varepsilon, x)) [o^\varepsilon, x] = \xi x.
$$
according to (F2), (A6), (A7), (A9) and P1(ii).
\end{proof}

\begin{corollary}
The set $\overline M$ is a complementary set to the subgroup $\overline H$, i. e.
for any $\alpha a \in {\overline G}$ there exist unique
${\overline \xi} \in {\overline H}$ and ${\overline x} \in {\overline M}$ such that
$ \alpha a = \overline \xi \cdot {\overline x}$.
\end{corollary}

\begin{proof}
The only values satisfying this relation are $\xi = \alpha, x = a.$
\end{proof}

To complete the proof of Theorem $\ref{th5}$, it remains to verify that
 the pair $f = (f_0, \, f_1)$ preserves all structural mappings, that is
$$
\hspace{13mm}
{{\overline a}}^{\overline \alpha} = \overline{a^\alpha}, \quad
{^{\overline a}} {\overline \alpha} = \overline{{^a}\alpha}, \quad
[{\overline a}, \, {\overline b}] = \overline{[a, \, b]}, \quad
({\overline a}, \, {\overline b}) = \overline{(a, \, b)}
\hspace{8mm} (F4)
$$
for any $\alpha \in H, \, a,b \in M$.

\begin{lemma}
\label{lm6}
For arbitrary elements $\alpha \in H, \, a \in M$
$$
{\overline a} \cdot {\overline \alpha} = (\varepsilon a) \cdot ((\alpha \cdot \theta) o) =
{^a}\alpha \, a^\alpha.
$$
\end{lemma}

\begin{proof}
Indeed, we have
$$
(\varepsilon a) \cdot ((\alpha \cdot \theta) o) =
(\varepsilon \cdot {^a}(\alpha \cdot \theta) \cdot (a^{\alpha \cdot \theta}, \, o)) [a^{\alpha \cdot \theta}, \, o] =
$$
$$
= ({^a}\alpha \cdot {^{a^\alpha}}\theta \cdot {^{a^{\alpha \cdot \theta}}}(\theta^{-1})) (a^{\alpha \cdot \theta})^{\theta^{-1}}  =
({^a}\alpha \cdot {^{a^\alpha}}(\theta \cdot \theta^{-1})) a^\alpha  = ({^a}\alpha \cdot \varepsilon) a^\alpha  =
{^a}\alpha \, a^\alpha
$$
according to (F2), (A1), (A11), (A10), P2(i), (A6).
\end{proof}

Using Lemmas $\ref{lm5}$ and $\ref{lm6}$, we get
$$
{\overline a} \cdot {\overline \alpha} =
{^a}\alpha \, a^\alpha = \overline{{^a}\alpha} \cdot \overline{a^\alpha}.
$$
On the other hand, according to (F1)
$$
{\overline a} \cdot {\overline \alpha} = {^{\overline a}}{\overline \alpha} \cdot {\overline a}^{\overline \alpha}.
$$
Hence, in view of the uniqueness of the corresponding decomposition in the exact product
${\overline G} = {\overline H} \odot {\overline M}$, we get
$$
{{\overline a}}^{\overline \alpha} = \overline{a^\alpha}, \quad
{^{\overline a}} {\overline \alpha} = \overline{{^a}\alpha}.
$$

Two last relations of (F4) are deduced in a similar manner.

\begin{lemma}
\label{lm7}
For any $a, b \in M$
$$
\overline{a} \cdot \overline{b} = (a, \, b) [a, \, b] =
\overline{(a, \, b)} \cdot \overline{[a, \, b]}.
$$
\end{lemma}

\begin{proof}
We have
$$
\overline{a} \cdot \overline{b} = \varepsilon a \cdot \varepsilon b =
(\varepsilon \cdot {^a}\varepsilon \cdot (a^\varepsilon, \, b)) [a^\varepsilon, \, b] =
(\varepsilon \cdot (a, \, b)) [a, \, b] =
(a, \, b) [a, \, b],
$$
according to (F2), (A7), P2(ii) and
$$
\overline{(a, \, b)} \cdot \overline{[a, \, b]} =
((a, b) \cdot \theta)o \cdot \varepsilon [a, b] =
((a, b) \cdot \theta) \cdot {^o}\varepsilon \cdot (o^\varepsilon, [a, b])[o^\varepsilon, [a, b]] =
$$
$$
= ((a, b) \cdot \theta) \cdot \varepsilon \cdot (o, [a, b])[o, [a, b]] =
(a, b) [a, b]
$$
according to (F2), (A6), (A7), (A9), P1(ii).
\end{proof}

By (F1)we have also
$$
\overline{a} \cdot \overline{b} = (\overline{a}, \, \overline{b}) \cdot [\overline{a} \cdot \overline{b}].
$$
Therefore
$$
[{\overline a}, \, {\overline b}] = \overline{[a, \, b]}, \quad
({\overline a}, \, {\overline b}) = \overline{(a, \, b)}.
$$
Theorem $\ref{th5}$ is proved.
\end{proof}


\section{Group actions}
\label{AG}

Let $M_H$ be an arbitrary hypergroup over the group and
$\Omega = (\Phi, \Psi, \Xi, \Lambda)$ be its system of structural mappings.
Here we use the results and denotations of Theorem $\ref{th5}$.

Consider an arbitrary map
$$
\kappa: M  \rightarrow H, \,\, a \mapsto \kappa_a
$$
and the set
$$
{\overline \kappa}{\overline M} = \{ {\overline \kappa_a} \cdot {\overline a} = \kappa_a a, \,\, a \in M \}.
$$
Obviously, this set is a subset of $\overline G$ complementary to the subgroup $\overline H$.
Therefore, in a standard way, we obtain a hypergroup over the group ${\overline \kappa}{\overline M}_{\overline H}$.
We denote its system of structural mappings by
${\overline \Omega}_\kappa = ({\overline \Phi}_\kappa, {\overline \Psi}_\kappa, {\overline \Xi}_\kappa, {\overline \Lambda}_\kappa)$
and use the following short notations:
$$
{\overline \Phi}_\kappa (\kappa_a a, {\overline \alpha}) =
(\kappa_a a)^{\overline \alpha}, \quad
{\overline \Psi}_\kappa (\kappa_a a, {\overline \alpha}) =
{^{\kappa_a a}}{\overline \alpha},
$$
$$
{{\overline \Xi}}_\kappa (\kappa_a a, \kappa_b b) =
[\kappa_a a, \kappa_b b], \quad
{{\overline \Lambda}}_\kappa (\kappa_a a, \kappa_b b) =
(\kappa_a a, \kappa_b b).
$$

Applying (F1), Lemma $\ref{lm5}$, homomorphity of $f_0$, we get
\begin{center}
${^{\kappa_a a}}{\overline \alpha} \cdot (\kappa_a a)^{\overline \alpha} =
(\kappa_a a) \cdot{\overline \alpha} =
{\overline \kappa_a} \cdot {\overline a} \cdot{\overline \alpha} =
{\overline \kappa_a} \cdot {^{\overline a}}{\overline \alpha}
\cdot {\overline a}^{\overline \alpha} =$
\end{center}
\begin{center}
$= {\overline \kappa_a} \cdot \overline{{^a}\alpha} \cdot \overline{\kappa_{a^\alpha}}^{-1}
\cdot \overline{\kappa_{a^\alpha}} \cdot \overline{a^\alpha} =
\overline{\kappa_a \cdot {^a}\alpha \cdot \kappa_{a^\alpha}^{-1}}
\cdot \kappa_{a^\alpha} a^\alpha$.
\end{center}

Similarly
$$
(\kappa_a a, \kappa_b b) \cdot [\kappa_a a, \kappa_b b] = \kappa_a a \cdot \kappa_b b =
(\kappa_a \cdot {^a}(\kappa_b) \cdot (a^{\kappa_b}, b)) [a^{\kappa_b}, b] =
$$
$$
= \overline{\kappa_a \cdot {^a}(\kappa_b) \cdot (a^{\kappa_b}, b)} \cdot \overline{[a^{\kappa_b}, b]} =
$$
$$
= \overline{\kappa_a \cdot {^a}(\kappa_b) \cdot (a^{\kappa_b}, b)} \cdot \overline{\kappa_{[a^{\kappa_b}, b]}}^{-1}
\cdot \overline{\kappa_{[a^{\kappa_b}, b]}} \cdot \overline{[a^{\kappa_b}, b]} =
$$
$$
= \overline{\kappa_a \cdot {^a}(\kappa_b) \cdot (a^{\kappa_b}, b) \cdot \kappa_{[a^{\kappa_b}, b]}^{-1}}
\cdot \kappa_{[a^{\kappa_b}, b]} [a^{\kappa_b}, b].
$$
Hence, the following lemma is proved.

\begin{lemma}
\label{lm8a}
The systems of structural mappings
${\overline \Omega} = ({\overline \Phi}, {\overline \Psi}, {\overline \Xi}, {\overline \Lambda})$ and
${\overline \Omega}_\kappa = ({\overline \Phi}_\kappa, {\overline \Psi}_\kappa, {\overline \Xi}_\kappa, {\overline \Lambda}_\kappa)$
of hypergroups over the groups
${\overline M}_{\overline H}$ and ${\overline \kappa}{\overline M}_{\overline H}$, respectively,
are connected by formulas:
\begin{itemize}
\item[(i)] $\quad (\kappa_a a)^{\overline \alpha} = \kappa_{a^\alpha} a^\alpha$,
\item[(ii)] $\quad {^{\kappa_a a}}{\overline \alpha}  =
\overline{\kappa_a \cdot {^a}\alpha \cdot \kappa_{a^\alpha}^{-1}}$,
\item[(ii)] $\quad  [\kappa_a a, \kappa_b b] = \kappa_{[a^{\kappa_b}, b]} [a^{\kappa_b}, b]$,
\item[(iv)] $\quad  (\kappa_a a, \kappa_b b) =
\overline{\kappa_a \cdot {^a}(\kappa_b) \cdot (a^{\kappa_b}, b) \cdot \kappa_{[a^{\kappa_b}, b]}^{-1}}$.
\end{itemize}
\end{lemma}

Consider a pair ${\overline f}_\kappa = ({\overline f}_0, {\overline f}'_\kappa)$,
consisting of the  isomorphism of groups
${\overline f}_0: H \rightarrow {\overline H}$
 defined in Theorem $\ref{th5}$
and of the bijection of sets
$$
{\overline f}'_\kappa: M \rightarrow {\overline \kappa}{\overline M}, \,\,\,
a \mapsto \kappa_a a = {\overline \kappa} \cdot {\overline a}.
$$
By using ${\overline f}_\kappa$ and
${\overline \Omega}_\kappa = ({\overline \Phi}_\kappa, {\overline \Psi}_\kappa,
{\overline \Xi}_\kappa, {\overline \Lambda}_\kappa)$,
we define a system of mappings
$\Omega_\kappa = (\Phi_\kappa, \Psi_\kappa, \Xi_\kappa, \Lambda_\kappa)$
in the following way:
$$
\Phi_\kappa: M \times H \rightarrow M, \quad \Phi_\kappa =
{\overline f}_\kappa \cdot {\overline \Phi}_\kappa \cdot ({\overline f}'_\kappa)^{-1},
$$
$$
\Psi_\kappa: M \times H \rightarrow H, \quad \Psi_\kappa =
{\overline f}_\kappa \cdot {\overline \Psi}_\kappa \cdot ({\overline f}_0)^{-1},
$$
$$
\Xi_\kappa: M \times M \rightarrow M, \quad \Xi_\kappa =
({\overline f}'_\kappa, {\overline f}'_\kappa) \cdot {\overline \Xi}_\kappa \cdot ({\overline f}'_\kappa)^{-1},
$$
$$
\Lambda_\kappa: M \times N \rightarrow H, \quad \Lambda_\kappa =
({\overline f}'_\kappa, {\overline f}'_\kappa) \cdot {\overline \Lambda}_\kappa \cdot ({\overline f}_0)^{-1}.
$$
This means that the pair ${\overline f}_\kappa = ({\overline f}_0, {\overline f}'_\kappa)$
preserves the systems of mappings
${\overline \Omega}_\kappa = ({\overline \Phi}_\kappa, {\overline \Psi}_\kappa,
{\overline \Xi}_\kappa, {\overline \Lambda}_\kappa)$ and
$\Omega_\kappa = (\Phi_\kappa, \Psi_\kappa, \Xi_\kappa, \Lambda_\kappa)$.

Applying the definition of mappings $(\Phi_\kappa, \Psi_\kappa, \Xi_\kappa, \Lambda_\kappa)$
to the  relations (i)-(iv) we get the following formulas:
\begin{itemize}
\item[$(i)_\kappa$] $\quad \Phi_\kappa (a, \alpha) = \Phi(a, \alpha)$,
\item[$(ii)_\kappa$] $\quad \Psi_\kappa (a, \alpha) = \kappa_a \cdot \Psi(a, \alpha) \cdot \kappa_{a^\alpha}^{-1}$,
\item[$(iii)_\kappa$] $\quad \Xi_\kappa (a, b) = \Xi(\Phi(a, \kappa_b), b)$,
\item[$(iv)_\kappa$] $\quad \Lambda_\kappa (a, b)=
\kappa_a \cdot \Psi(a, \kappa_b) \cdot \Lambda(\Phi(a. \kappa_b), b) \cdot \kappa_{\Xi(\Phi(a, \kappa_b), b)}^{-1}$.
\end{itemize}
We can define the system of mappings $\Omega_\kappa = (\Phi_\kappa, \Psi_\kappa, \Xi_\kappa, \Lambda_\kappa)$
by these formulas.

\begin{proposition}
\label{prop3}
Let $M_H$ be an arbitrary hypergroup over the group and
$\Omega = (\Phi, \Psi, \Xi, \Lambda)$ be its system of structural mappings.
Let $f_0: H \rightarrow H'$ be an arbitrary epimorphism of groups,
$f_1: M \rightarrow M'$ be an arbitrary surjection of sets and
$\Omega' = (\Phi', \Psi', \Xi', \Lambda')$ be a system of mappings on the pair $(H', M')$
such that the pair $f = (f_0, f_1)$ preserves the systems of mappings $\Omega$ and $\Omega'$.
Then the system of mappings $\Omega'$ satisfies conditions P1(ii), P2, P3, P4.
Condition P1(i) holds partially: any equation $[x', a'] = b'$ in $M'$ has a root.
Condition P1(i) is fulfilled completely, if $f_1$ is a bijection.
\end{proposition}

\begin{proof}
Conditions P2(i), P4 for $\Omega'$ are identities in variables belonging to $M'$ and $H'$,
formed by superpositions of mappings of the system $\Omega'$.
They are the images of corresponding identities in variables belonging to $M$ and $H$
(formed by superpositions of mappings of system $\Omega$)
which are satisfied, since $M_H$ is a hypergroup over the group.
Since $f_0$ is an epimorphism, $f_1$ is a surjection and their pair $f$ preserves the system of
mappings $\Omega$ and $\Omega'$, conditions P2(ii), P4 are satisfied for the system of mappings
$\Omega'$, as well.

 Property P3 for $\Omega'$ follows from the surjectivity of mappings
 $\Psi, f_0, f$.

 The image of the neutral element $\varepsilon$ of the group $H$ under the epimorphism $f_0$
 is the neutral element $\varepsilon'$ of the group $H'$.
 Therefore condition P2(ii) is satisfied for $\Omega'$.

 Condition P1(ii) will be satisfied for $\Omega'$, since $o' = f_1(o)$ is
 the neutral element of the binary operation $\Xi'$.

 Finally, any equation $[x', a'] = b'$ in $M'$ has a root, since
 for arbitrary $a, b \in M$ such that $f_1(a) = a', \, f_1(b) = b'$
 the equation $[x, a] = b$ has a root in $M$.
 If $f_1$ is a bijection, it is clear that this root is unique.
\end{proof}

\begin{corollary}
For an arbitrary hypergroup over the group $M_H$ with the system of structural mappings
$\Omega = (\Phi, \Psi, \Xi. \Lambda)$ and for any map $\kappa: M \rightarrow H$
the system of mappings
$\Omega_\kappa = (\Phi_\kappa, \Psi_\kappa, \Xi_\kappa, \Lambda_\kappa)$,
defined by formulas $(i)_\kappa-(iv)_\kappa$ defines a structure of
a hypergroup over the group H on the basic set $M$.
\end{corollary}

We denote this hypergroup $(M_H)_\kappa$.

\begin{proof}
The pair of maps ${\overline f}_\kappa = ({\overline f}_0, {\overline f}'_\kappa)$
satisfies conditions of the Proposition $\ref{prop3}$.
Consequently, the system $\Omega_\kappa$ satisfies the conditions P1(ii), P2, P3, P4.
Condition P1(i) is also satisfied, since ${\overline f}'_\kappa$
has an inverse map ${\overline f}'_{\kappa^{-1}}$.
Hence, the system of mappings $\Omega_\kappa$ defines a structure of
a hypergroup over the group $H$  on $M$.
\end{proof}

Let $H$ be a group, $M$ be an arbitrary set, $Hg(M, H)$
be the set of all
hypergroups $M_H$ with the basic set $M$ over the group $H$.
This means that for a given pair $(M, H)$,
$Hg(M, H)$ is the set of all (structural) mappings
$\Omega = (\Phi, \Psi, \Xi, \Lambda)$  satisfying
conditions P1 - P4.

Consider the set
$$
H^M  = \{  \kappa: M  \rightarrow H, \,\, a \mapsto \kappa_a \}
$$
with a group operation $\ast$ induced
by the group operation on $H$:
for any elements $\kappa, \lambda \in H^M$ by definition
$$
(\kappa * \lambda)_a = \kappa_a \cdot \lambda_a.
$$

For any element $\kappa \in H^M$ and
for an arbitrary map $f: M \rightarrow M$,
the composite $f \cdot \kappa$ is an element of $H^M$
such that $(f \cdot \kappa)_a = \kappa_{f(a)}$.

\begin{proposition}
\label{prop4a}
(a) The neutral element for the operation $\ast$
is the map $\varepsilon ^M$,
sending each element $a \in M$ to the neutral
element $\varepsilon$
of the group $H$.

(b) The inverse element to $\kappa \in H^M$ is an element
$\kappa^{-1}$ such that $(\kappa^{-1})_a = (\kappa_a)^{-1}$
for every $a \in M$.

(c) For any elements $\kappa, \lambda \in H^M$
and for an arbitrary map $f: M \rightarrow M$
we have the relations:

$$
(f\cdot \kappa) * (f \cdot \lambda) = f \cdot (\kappa * \lambda), \quad
f \cdot \varepsilon^M = \varepsilon^M.
$$
\end{proposition}

The proof is trivial.

\begin{theorem}
\label{th6}
Let $H$ be an arbitrary group, $M$ be an arbitrary set.
Then the mapping
$$
T: H^M \times Hg(M, H) \rightarrow  Hg(M, H), \quad (\kappa, M_H) \mapsto (M_H)_\kappa
$$
is a reverse action of the group $H^M$ on the set $Hg(M, H)$, that is, the following properties hold:
\begin{itemize}
\item[(a)] for arbitrary $\kappa, \lambda \in H^M$  and for every hypergroup $M_H$ over the group
$(M_H)_{\lambda * \kappa} = ((M_H)_\kappa)_\lambda$;
\item[(b)] the neutral element $\varepsilon^M$ of the group $H^M$ acts on $Hg(M, H)$ identically,
that is $(M_H)_{\varepsilon^M} = M_H$  for any $M_H \in Hg(M, H)$.
\end{itemize}
\end{theorem}

\begin{proof}
To get the properties (a) and (b) it is necessary to check
the corresponding properties for each structural mapping
separately,
using the definitions $(i)_\kappa- (iv)_\kappa$.
This is trivial for property (b).
Let us check property (a).

$(i)_\kappa$. For structural mapping $\Phi$ we have
$$
((M_H)_\kappa)_\lambda = (M_H)_\kappa = M_H = (M_H)_{\lambda * \kappa}.
$$

$(ii)_\kappa$. Let us check property (a) for $\Psi$. For any $\kappa, \lambda \in H^M$ we have
$$
(\Psi_\kappa)_\lambda (a, \alpha) = \lambda_a \cdot \Psi_\kappa(a, \alpha) \cdot \lambda_{a^\alpha}^{-1} =
\lambda_a \cdot \kappa_a \cdot \Psi(a, \alpha) \cdot \kappa_{a^\alpha}^{-1} \cdot \lambda_{a^\alpha}^{-1} =
$$
$$
= (\lambda * \kappa)_a \cdot \Psi(a, \alpha) \cdot (\lambda * \kappa)_{a^\alpha}^{-1} =
\Psi_{\lambda * \kappa} (a, \alpha).
$$

$(iii)_\kappa$. In the case of the structural mapping
$\Xi$ we have
$$
(\Xi_\kappa)_\lambda (a, b) = \Xi_\kappa(\Phi_\kappa(a, \lambda_b), b) =
\Xi(\Phi(\Phi(a, \lambda_b), \kappa_b), b) =
$$
$$
= \Xi(\Phi(a, \lambda_b \cdot \kappa_b), b)
= \Xi(\Phi(a, (\lambda * \kappa)_b), b)= \Xi_{\lambda * \kappa} (a, b).
$$

$(iv)_\kappa$. For the structural mapping $\Lambda$
the calculations are more complicated.
$$
(\Lambda_\kappa)_\lambda (a, b)=
\lambda_a \cdot \Psi_\kappa(a, \lambda_b) \cdot \Lambda_\kappa(\Phi_\kappa(a. \lambda_b), b) \cdot
\lambda_{\Xi_\kappa(\Phi_\kappa(a, \lambda_b), b)}^{-1} =
$$
$$
= \lambda_a \cdot (\kappa_a \cdot \Psi(a, \lambda_b) \cdot \kappa_{\Phi(a, \lambda_b)}^{-1}) \cdot
\kappa_{\Phi(a, \lambda_b)} \cdot \Psi(\Phi(a, \lambda_b), \kappa_b) \cdot
$$
$$
\cdot \Lambda(\Phi(\Phi(a, \lambda_b), \kappa_b), b) \cdot \kappa_{\Xi(\Phi(\Phi(a, \lambda_b), \kappa_b), b)}^{-1}
\cdot \lambda_{\Xi(\Phi(\Phi(a, \lambda_b), \kappa_b), b)}^{-1} =
$$
$$
= (\lambda_a \cdot \kappa_a) \cdot \Psi(a, \lambda_b) \cdot  \Psi(\Phi(a, \lambda_b), \kappa_b) \cdot
\Lambda(\Phi(a, \lambda_b \cdot \kappa_b), b) \cdot
$$
$$
(\lambda_{\Xi(\Phi(a, \lambda_b \cdot \kappa_b), b)} \cdot \kappa_{\Xi(\Phi(a, \lambda_b \cdot \kappa_b), b)})^{-1} =
$$
$$
= (\lambda * \kappa)_a \cdot \Psi(a, (\lambda * \kappa)_b) \cdot \Lambda(\Phi(a, (\lambda * \kappa)_b), b) \cdot
(\lambda * \kappa)_{\Xi(\Phi(a, (\lambda * \kappa)_b), b)}^{-1} =
$$
$$
= \Lambda_{\lambda * \kappa} (a, b).
$$
Here we have also used the relation (A1).
\end{proof}

Now, we apply proposition $\ref{prop3}$ to a particular case, when
\begin{center}
$M' = M, \quad H' = H, \quad f_0 = 1_H$ and $f_1$ is a bijection.
\end{center}
In this case we use the following notations: $f = f_1^{-1}$ and
\begin{itemize}
\item[$(i)_f$] $\quad \Phi_f = (f^{-1}, 1_H) \cdot \Phi \cdot f$,
\item[$(ii)_f$] $\quad  \Psi_f = (f^{-1}, 1_H) \cdot \Psi \cdot 1_H$,
\item[$(iii)_f$] $\quad  \Xi_f = (f^{-1}, f^{-1}) \cdot \Xi \cdot f$,
\item[$(iv)_f$] $\quad  \Lambda_f = (f^{-1}, f^{-1}) \cdot \Lambda \cdot 1_H$.
\end{itemize}

The set of all bijections $M \rightarrow M$ is denoted by $S_M$.

Then as an immediate consequence of proposition $\ref{prop3}$ we have
the following result.

\begin {corollary}
For any $M_H \in Hg(N, H)$ and for each $f \in S_M$
the system of mappings $\Omega_f, = (\Phi_f, \Psi_f, \Xi_f, \Lambda_f)$
defines a structure of a hypergroup over the group $H$ on the basic set $M$.
\end{corollary}
We denote this hypergroup over the group by $(M_H)_f$.

It is not difficult to check the following theorem.

\begin{theorem}
\label{thm7a}
For an arbitrary group $H$ and for each set $M$ the mapping
$$
{\mathcal I}: Hg(M, H) \times S_M \rightarrow Hg(M, H), \quad (M_H, f) \mapsto (M_H)_f
$$
defines an action of $S_M$ on the set $Hg(M, H)$,
i.e., for any $M_H$ from $Hg(M, H)$ we have:
\begin{center}
(a) $(M_H)_{f \cdot g} = ((M_H)_f)_g, \quad f, g \in S_M$, \quad
(b) $(M_H)_{1_M} = M_H.$
\end{center}
\end {theorem}

The following lemma connects the actions $T$ and $\mathcal I$.

\begin{lemma}
\label{lem11a}
The actions $T$ and $\mathcal I$ are connected
by the following relation:
for an arbitrary hypergroup over the group $M_H \in Hg(M, H)$
and for any elements $\kappa \in M^H$ and $f \in S_M$
$$
((M_H)_f)_\kappa = ((M_H)_{f \cdot \kappa})_f.
$$
\end{lemma}

\begin{proof}
For the systems of structural mappings
\begin{center}
$(\Omega_f)_\kappa = ((\Phi_f)_\kappa, (\Psi_f)_\kappa, (\Xi_f)_\kappa, (\Lambda_f)_\kappa)$ and
$(\Omega_{f \cdot\kappa)_f} = ((\Phi_{f \cdot\kappa})_f, (\Psi_{f \cdot\kappa})_f, (\Xi_{f \cdot\kappa})_f, (\Lambda_{f \cdot\kappa})_f)$
\end{center}
of hypergroups over the group $((M_H)_f)_\kappa$ and $((M_H)_{f \cdot\kappa})_f$, respectively,
we need to check  the equality of all corresponding components.
To this end we use the definitions of the actions $T$ and $\mathcal I$.

1) $(\Phi_f)_\kappa (a, \alpha) = \Phi_f (a, \alpha) = ((f^{-1}. 1_H) \cdot \Phi \cdot f) (a, \alpha) =
f(f^{-1}(a)^\alpha)$,

$(\Phi_{f \cdot \kappa})_f (a, \alpha) = ((f^{-1}, 1_H) \cdot \Phi_{f \cdot \kappa} \cdot f)(a, \alpha) =
f(\Phi_{f \cdot \kappa}(f^{-1}(a), \alpha)) =$\\
$= f(f^{-1}(a)^\alpha)$.

2) $(\Psi_f)_\kappa (a, \alpha) = \kappa_a \cdot \Psi_f (a, \alpha) \cdot \kappa^{-1}_{\Phi_f(a, \alpha)} =
\kappa_a \cdot {^{f^{-1}(a)}} \alpha \cdot \kappa^{-1}_{f(f^{-1}(a)^\alpha)}$,

$(\Psi_{f \cdot \kappa})_f (a, \alpha) = ((f^{-1}, 1_H) \cdot \Psi_{f \cdot \kappa} \cdot 1_H)(a, \alpha) =
\Psi_{f \cdot \kappa} (f^{-1}(a), \alpha) = $\\
$= (f \cdot \kappa)_{f^{-1}(a)} \cdot \Psi (f^{-1}(a), \alpha) \cdot ((f \cdot \kappa)_{\Phi(f^{-1}(a), \alpha)})^{-1} =$\\
$\kappa_a \cdot {^{f^{-1}(a)}} \alpha \cdot \kappa^{-1}_{f(f^{-1}(a)^\alpha)}$.

3) $(\Xi_f)_\kappa (a, b) = \Xi_f (\Phi_f(a, \kappa_b), b) = f([f^{-1}(f(f^{-1}(a)^{\kappa_b}), f^{-1}(b)]) =$\\
$= f([f^{-1}(a)^{\kappa_b}, f^{-1}(b)])$,

$(\Xi_{f \cdot \kappa})_f (a, b) = f(\Xi_{f \cdot \kappa}(f^{-1}(a), f^{-1}(b))) = $\\
$= f([f^{-1}(a)^{(f \cdot \kappa)_{f^{-1}(b)}}, f^{-1}(b)]) = f([f^{-1}(a)^{\kappa_b}, f^{-1}(b)])$.

4) $(\Lambda_f)_\kappa (a, b) = \kappa_a \cdot \Psi_f(a, \kappa_b) \cdot \Lambda_f(\Phi_f (a, \kappa_b), b) \cdot \kappa^{-1}_{\Xi_f (\Phi_f (a, \kappa_b), b)} =$\\
$\kappa_a \cdot {^{f^{-1}(a)}}(\kappa_b) \cdot (f^{-1}(f(f^{-1}(a)^{\kappa_b})), f^{-1}(b)) \cdot \kappa^{-1}_{f([f^{-1}(f(f^{-1}(a)^{\kappa_b})), f^{-1}(b)])} =$\\
$= \kappa_a \cdot {^{f^{-1}(a)}}(\kappa_b) \cdot (f^{-1}(a)^{\kappa_b}, f^{-1}(b)) \cdot \kappa^{-1}_{f([f^{-1}(a)^{\kappa_b}, f^{-1}(b)])}$,

$(\Lambda_{f \cdot \kappa})_f (a, b) = \Lambda_{f \cdot \kappa} (f^{-1}(a), f^{-1}(b)) = $\\
$= (f \cdot \kappa)_{f^{-1}(a)}  \cdot \Psi(f^{-1}(a), (f \cdot \kappa)_{f^{-1}(b)})\cdot \Lambda(\Phi(f^{-1}(a),(f \cdot \kappa)_{f^{-1}(b)}, f^{-1}(b)) \cdot $\\
$((f \cdot \kappa)_{\Xi (\Phi(f^{-1}(a), (f \cdot \kappa)_{f^{-1}(b)}), f^{-1}(b))})^{-1} =$\\
$= \kappa_a \cdot {^{f^{-1}(a)}}(\kappa_b) \cdot (f^{-1}(a)^{\kappa_b}, f^{-1}(b)) \cdot \kappa^{-1}_{f([f^{-1}(a)^{\kappa_b}, f^{-1}(b)])}$.

Lemma $\ref{lem11a}$ is proved.
\end{proof}

\begin{proposition}
\label{pr5}
The set $H^M \times S_M$ together with the binary operation
$$
\kappa f \cdot \lambda g = ((f \cdot \lambda) * \kappa)(f \cdot g)
$$
is a group.
\end{proposition}

 We denote this group by $\mathcal G$: $\mathcal G$ is the {\it wreath product} of
 the group $H$ by $S(M)$.

\begin{proof}
The above defined operation is associative:
$$
(\kappa f \cdot \lambda g) \cdot \mu h = ((f \cdot \lambda) * \kappa)(f \cdot g) \cdot \mu h =
(((f \cdot g)\cdot \mu) * ((f \cdot \lambda) * \kappa))((f \cdot g) \cdot h) =\\
$$
$$
= (((f \cdot (g\cdot \mu)) * (f \cdot \lambda)) * \kappa)(f \cdot (g \cdot h))
= ((f \cdot ((g\cdot \mu) * \lambda))) * \kappa)(f \cdot (g \cdot h)) =\\
$$
$$
= \kappa f \cdot ((g \cdot \mu) * \lambda)(g \cdot h) = \kappa f \cdot (\lambda g \cdot \mu h).
$$

The neutral element of this operation is $\varepsilon^M 1_M$.  Indeed
$$
\kappa f \cdot \varepsilon^M 1_M = (\kappa * (f \cdot \varepsilon^M))(f \cdot 1_M) =
(\kappa * \varepsilon^M)f = \kappa f,\\
$$
$$
\varepsilon^M 1_M \cdot \kappa f= (\varepsilon^M * (1_M \cdot \kappa))(1_M \cdot f)= \kappa f.
$$

The inverse element to $\kappa f$ is $(f^{-1} \cdot \kappa^{-1})f^{-1}$. Indeed
$$
\kappa f \cdot (f^{-1} \cdot \kappa^{-1})f^{-1} =
(\kappa * (f \cdot (f^{-1} \cdot \kappa^{-1}))(f \cdot f^{-1}) = \varepsilon^M 1_M,
$$
$$
(f^{-1} \cdot \kappa^{-1})f^{-1} \cdot \kappa f =
((f^{-1} \cdot \kappa^{-1}) * (f^{-1} \cdot \kappa))(f^{-1} \cdot f) =
(f^{-1} \cdot (\kappa^{-1} * \kappa)) 1_M = \\
$$
$$
= (f^{-1} \cdot \varepsilon^M) 1_M = \varepsilon^M 1_M.
$$
The proposition is proved.
\end{proof}

\begin{theorem}
\label{the8}
The mapping
$$
{\mathcal A}: Hg(M, H) \times {\mathcal G} \rightarrow Hg(M, H), \quad
(M_H, \kappa f) \mapsto (M_H)_{\kappa f} = ((M_H)_\kappa)_f
$$
is an action of $\mathcal G$ on the set $Hg(M, H)$,
that is the following relations hold:
\begin{center}
(a) $(M_H)_{\kappa f \cdot \lambda g} = ((M_H)_{\kappa f})_{\lambda g}$, \quad
(b) $(M_H)_{\varepsilon^M 1_M} = M_H$.
\end{center}
\end{theorem}

\begin{proof}
 Using Theorems $\ref{th6}$, $\ref{thm7a}$,
 the lemma $\ref{lem11a}$ and Proposition $\ref{pr5}$ we get:
$$
(M_H)_{\kappa f \cdot \lambda g} = (M_H)_{((f \cdot \lambda) * \kappa)(f \cdot g)} =
((M_H)_{(f \cdot \lambda) *\kappa})_{f \cdot g} =
((((M_H)_\kappa)_{f \cdot \lambda})_f)_g = \\
$$
$$
= ((((M_H)_\kappa)_f)_\lambda)_g = ((M_H)_{\kappa f})_{\lambda g}.
$$
$$
(M_H)_{\varepsilon^M 1_M} = ((M_H)_{\varepsilon^M})_{1_M} = (M_H)_{1_M} = M_H.
$$
Thus, Theorem $\ref{the8}$ is proved.
\end{proof}

For any hypergroup over the group $M_H \in Hg(M, H)$,
the set of all hypergroups
 $(M_H)_{\kappa f}$ over the group, where
 $\kappa f \in {\mathcal G}$,
forms an {\it orbit} of $M_H$ under the action
$\mathcal A$ of $\mathcal G$ on the set $Hg(M, H)$.
The orbits having a common element coincide.
The set of all orbits under the action $\mathcal A$
of $\mathcal G$ on the set $Hg(M, H)$
is denoted by $Hg(M, H)\slash {\mathcal A}$ or
$Hg(M, H)\slash {\mathcal G}$.


\section{Proof of the main theorem}

Let $H$ be an arbitrary group.
We call an {\it extension} of $H$
any monomorphism of groups
$\varphi: H \rightarrow G$.
The cardinality (cardinal number)
of the quotient-set $G \slash \varphi(H)$
(or equivalently $\varphi(H) \backslash G$)
is called the {\it degree} of extension $\varphi$.
Two extensions $\varphi: H \rightarrow G$ and
$\varphi': H \rightarrow G'$
are called {\it isomorphic}, if there exists
an isomorphism of groups
$\pi: G \rightarrow G'$ such that
$\varphi' = \varphi \cdot \pi$.
Evidently, isomorphic extensions have the same degree
of extension.
We denote by $Ext_m (H)$ the set of all classes
of isomorphic extensions
of the group $H$, having a degree of extension $m$.

Suppose that $M$ is an arbitrary set of cardinality $m$.
The  mappings
$$
{\mathcal H}: Ext_m(H) \rightarrow Hg(M, H) {\slash} {\mathcal A}, \quad
{\mathcal E}: Hg(M,H) {\slash} {\mathcal A} \rightarrow Ext_m(H)
$$
are defined as follows.

Let $\varphi: H \rightarrow G$ be an arbitrary extension
of the group $H$ of a degree $m$ and
$\overline H$ be its image. We denote by
$\overline \varphi$ the corestriction of $\varphi$
to $\overline H$:
$$
\overline{\varphi}: H \rightarrow {\overline H}, \,\, x \mapsto \varphi(x).
$$

Let $\overline M$ be a complementary set
to $\overline H$ in $G$,
${\overline M}_{\overline{H}}$ be
the associated hypergroup over the group.
Consider an arbitrary set $M$ of cardinality
$m = |M|$ and
an arbitrary bijection
${\overline \varphi}_M: M \rightarrow {\overline M}$.
We set
${\tilde \varphi} = ({\overline \varphi},
{\overline \varphi}_M)$ and
consider the unique hypergroup over the group
$M_H$, which is
isomorphic to ${\overline M}_{\overline H}$
under $\tilde \varphi$
(Proposition \ref{prop3}).
Now we define $\mathcal H$ by assigning
to each extension $\varphi$ this
hypergroup over the group $M_H$.

Conversely, let $M_H$ be an arbitrary
hypergroup over the group
belonging to an element (orbit) of
$Hg(M, H) \slash {\mathcal A}$.
Consider the outer exact product
${\overline G} = H \odot M$,
and the corresponding canonical map
$f_0: H \rightarrow {\overline G}$ (see section 3).
We define $\mathcal E$ by assigning
to each $M_H$ the extension $\varphi = f_0$.

\begin{theorem}
\label{thrm9}
The mappings $\mathcal H$ and $\mathcal E$
are well defined.
\end{theorem}

\begin{proof}
We need to check the following two assertions.

1) The orbit under action $\mathcal A$
of a hypergroup $M_H$,
corresponding to an extension
$\varphi: H \rightarrow G$
 via above construction,  does not depend:
\begin{itemize}
\item[(a)] on the choice of a complementary set ${\overline M}$ to $\overline H$ in $\overline G$,
\item[(b)] on the choice of a bijection ${\overline \varphi}_M: M \rightarrow {\overline M}$ and
\item[(c)] on the choice of a representative $\varphi: H \rightarrow G$ in a class of isomorphic extensions.
\end{itemize}

2) For any hypergroup over the group
$M_H \in Hg(M, H)$ and
for an arbitrary element $\kappa f \in {\mathcal G}$
the extensions $f_0$ and $(f_0)_{\kappa f}$,
corresponding to hypergroups $M_H $ and
$(M_H)_{\kappa f}$,
are isomorphic.

Below we prove these assertions.

1a) Assume that we have arbitrary two complementary sets
$\overline M$
and ${\overline M}'$ to $\overline H$ in $G$.
Then there is a (unique) $\kappa \in H^M$ such that
${\overline M}' = {\overline \kappa} \cdot {\overline M}$.
Therefore, if $M_H, \, (M_H)' \in Hg(M, H)$
correspond to
complementary sets $\overline M$ and
${\overline M}'$, respectively,
then $(M_H)' = (M_H)_\kappa$.
Consequently, $M_H$ and $(M_H)'$ belong
to a same orbit under action
$\mathcal A$ of the group $\mathcal G$.

1b) Assume that
${\overline \varphi}_M: M \rightarrow {\overline M}$ and
${\overline \varphi}'_M: M \rightarrow {\overline M}$
are two bijections.
Then there exists a (unique) element $f \in S_M$ such that
${\overline \varphi}'_M = f \cdot {\overline \varphi}_M$.
This means that the hypergroups
$M_H, \, (M_H)' \in Hg(M, H)$,
corresponding to the complementary
set $\overline M$ under bijections
${\overline \varphi}_M$ and
${\overline \varphi}'_M$, are connected by relation
$(M_H)' = (M_H)_f$.
Therefore $M_H$ and $(M_H)'$ belong to a same orbit
under action $\mathcal A$
of the group $\mathcal G$.

1c) Let $\varphi: H \rightarrow G$ and
$\varphi': H \rightarrow G'$
be two isomorphic extensions via
an isomorphism $\pi: G \rightarrow G'$.
Let $\overline M$ be a complementary set to
${\overline H} = \varphi (H)$ in $G$.
Then ${\overline M}' = \pi({\overline M})$
is a complementary set to
${\overline H}' = \varphi' (H)$ in $G'$.

Let
${\overline \varphi}_M: M \rightarrow {\overline M}$
be a bijection.
Then
${\overline \varphi}'_M = {\overline \varphi}_M \cdot \pi: M \rightarrow {\overline M}'$
is a bijection as well.
Assume that $M_H$ corresponds to the hypergroup
${\overline M}_{\overline H}$
with respect to ${\overline \varphi}_M$ and
$(M_H)'$ corresponds to the hypergroup
${\overline M}'_{{\overline H}'}$
with respect to ${\overline \varphi}'_M$.
Then it is not difficult to see that $M_H = (M_H)'$.

So, the assertion 1) is proved. Now we check the assertion 2).

2)
The extension
$\varphi = f_0: H \rightarrow {\overline G}$
associated with the hypergroup $M_H$
is defined as follows.
The group ${\overline G} = H \odot M$
has a basic set
$H \times M$ and a binary operation
$$
\alpha a \cdot \beta b = (\alpha \cdot {^a}\beta \cdot (a^\beta, b))[a^\beta, b].
$$
The monomorphism $f_0$ sends an element
$\alpha$ of $H$
to the element $(\alpha \cdot \theta)o$
of the group
$\overline G$, where
$o$ is the left neutral element of the right quasigroup $(M, \Xi)$
and $\theta = (o, o)^{-1}$.

By definition
$(M_H)_{\kappa \cdot f} = ((M_H)_\kappa)_f$.
Therefore,it is sufficient to check  the assertion 2) separate;y
for $(M_H)_\kappa$ and $(M_H)_f$.
We start by investigating the case $(M_H)_\kappa$.

The extension $(f_0)_\kappa: H \rightarrow {\overline G}_\kappa$,
associated with the hypergroup $(M_H)_\kappa$,
is defined in the following way.

The group ${\overline G}_\kappa = (H \odot M)_\kappa$ has a basic set
$H \times M$ as well,
but the binary operation is defined by formula
$$
(\alpha a \cdot \beta b)_\kappa =
(\alpha \cdot ({^a}\beta)_\kappa \cdot ((a^\beta)_\kappa, b)_\kappa)[(a^\beta)_\kappa, b]_\kappa.
$$

Here for brevity we use notations
$$
\Phi_\kappa (a, \alpha) = (a^\alpha)_\kappa, \quad
\Psi_\kappa (a, \alpha) = ({^a}\alpha)_\kappa,
$$
$$
\Xi_\kappa (a, b) = [a, b]_\kappa, \quad
\Lambda_\kappa (a, b) = (a, b)_\kappa.
$$
Note that then the relations
$(i)_\kappa - (iv)_\kappa$ of section 4 take a form
\begin{itemize}
\item[$(i)'_\kappa$] \quad $(a^\alpha)_\kappa  = a^\alpha$,
\item[$(ii)'_\kappa$] \quad $({^a}\alpha)_\kappa = \kappa_a \cdot {^a}\alpha \cdot \kappa_{a^\alpha}^{-1}$,
\item[$(iii)'_\kappa$] \quad $[a, b]_\kappa = [a^{\kappa_b}, b]$,
\item[$(iv)'_\kappa$] \quad $(a, b)_\kappa = \kappa_a \cdot {^a}(\kappa_b) \cdot
(a^{\kappa_b}, b) \cdot \kappa_{[a^{\kappa_b}, b]}^{-1}$.
 \end{itemize}

 Applying these relations, we get
$$
(\alpha \cdot ({^a}\beta)_\kappa \cdot ((a^\beta)_\kappa, b)_\kappa)[(a^\beta)_\kappa, b]_\kappa =
$$
$$
= (\alpha \cdot
\kappa_a \cdot {^a}\beta \cdot \kappa_{a^\beta}^{-1} \cdot
\kappa_{a^\beta} \cdot {^{a^\beta}}(\kappa_b) \cdot ((a^\beta)^{\kappa_b}, b) \cdot \kappa_{[(a^\beta)^{\kappa_b}, b]}^{-1})
[(a^\beta)^{\kappa_b}, b] =
$$
$$
= (\alpha \cdot \kappa_a \cdot {^a}(\beta \cdot \kappa_b) \cdot
(a^{\beta \cdot \kappa_b}, b) \cdot \kappa_{[a^{\beta \cdot \kappa_b}, b]}^{-1})
[a^{\beta \cdot \kappa_b}, b].
$$

Consider a map of sets
$$
g_\kappa: H \times M \rightarrow H \times M, \quad
\alpha a \mapsto (\alpha \cdot \kappa_a)a.
$$
This map represents a group homomorphism from
${\overline G}_\kappa$ to $\overline G$ since
$$
g_\kappa (\alpha a \cdot \beta b)_\kappa =
g_\kappa (\alpha \cdot \kappa_a \cdot {^a}(\beta \cdot \kappa_b) \cdot
(a^{\beta \cdot \kappa_b}, b) \cdot \kappa_{[a^{\beta \cdot \kappa_b}, b]}^{-1})
[a^{\beta \cdot \kappa_b}, b]) =
$$
$$
= (\alpha \cdot \kappa_a \cdot {^a}(\beta \cdot \kappa_b) \cdot
(a^{\beta \cdot \kappa_b}, b)) [a^{\beta \cdot \kappa_b}, b] =
(\alpha \cdot \kappa_a)a \cdot (\beta \cdot \kappa_b)b = g_\kappa(a) \cdot g_\kappa(b).
$$
Furthermore, $g_\kappa$ is a bijection,
since it has an inverse map
$$
(g_\kappa)^{-1} = g_{\kappa^{-1}}, \quad (\kappa^{-1})_a = \kappa_a^{-1}.
$$
Therefore $g_\kappa$ defines an group isomorphism.

So, we have proved the following lemma.

\begin{lemma}
\label{lem10}
For any hypergroup $M_H$ over the group and
for every element $\kappa \in H^M$
there exists a canonical group isomorphism $g_\kappa$
from outer product associated with the hypergroup
$(M_H)_\kappa$
to outer product associated with the hypergroup $M_H$.
This isomorphism is given by the formula
$g_\kappa (\alpha a) = (\alpha \cdot \kappa_a)a$.
\end{lemma}

The monomorphism $(f_0)_\kappa$ sends
an element $\alpha$ of $H$
to the element $(\alpha \cdot \theta_\kappa) o_\kappa$
of the group ${\overline G}_\kappa$,
where $o_\kappa$ is the left neutral element
of the right quasigroup $(M, \Xi_\kappa)$ and
$\theta_\kappa = (o_\kappa, o_\kappa)_\kappa^{-1}$.

\begin{lemma}
\label{lem11}
For any hypergroup $M_H$ over the group and
for every element $\kappa \in H^M$
$$
(a) \,\, o_\kappa = o, \quad
(b) \,\, \theta_\kappa = \theta \cdot \kappa_o^{-1}, \quad
(c) \,\, (f_0)_\kappa \cdot g_\kappa = f_0
$$
\end{lemma}

\begin{proof}
(a) According to $(iii)'_\kappa$, A7 and P1(ii)
for an arbitrary element $a \in M$ we have
$$
[o, a]_\kappa = [o^{\kappa_a}, a] = [o, a] = a.
$$
Hence $o_\kappa = o$.

(b) According to (a), $(iv)'_\kappa$,
A8, A7, A9, P1(ii)
$$
\theta_\kappa = (o_\kappa, o_\kappa)_\kappa^{-1} =
(\kappa_o \cdot {^o}(\kappa_o) \cdot (o^{\kappa_o}, o) \cdot \kappa_{[o^{\kappa_o}, o]}^{-1})^{-1} =
$$
$$
= (\kappa_o \cdot \theta^{-1} \cdot \kappa_o \cdot \theta \cdot \theta^{-1} \cdot \kappa_o^{-1})^{-1} =
\theta \cdot \kappa_o^{-1}.
$$

(c) According to the definition of $(f_o)_\kappa$, (a), (b)
$$
(f_o)_\kappa (\alpha) = (\alpha \cdot \theta_\kappa) o_\kappa = (\alpha \cdot \theta \cdot \kappa_o^{-1})o,
$$
Therefore
$$
g_\kappa (f_o)_\kappa (\alpha)) = g_\kappa ((\alpha \cdot \theta \cdot \kappa_o^{-1})o)=
(\alpha \cdot \theta)o = f_0 (\alpha).
$$
Lemma $\ref{lem11}$ is proved.
\end{proof}

Thus, we checked the assertion 2) for $(M_H)_\kappa$.

Now we investigate the case of $(M_H)_f$.

The extension $(f_0)_f: H \rightarrow {\overline G}_f$,
associated with the hypergroup $(M_H)_f$,
is defined in the following.
The group ${\overline G}_f = (H \odot M)_f$ has also
the basic set $H \times M$,
its binary operation is defined by the formula:
$$
\hspace{20mm}
(\alpha a \cdot \beta b)_f = (\alpha \cdot ({^a}\beta)_f \cdot ((a^\beta)_f, b)_f)[(a^\beta)_f, b]_f.
\hspace{15mm}  (F5)
$$
The monomorphism $(f_0)_f$ is defined by assigning to each element $\alpha \in H$
an element $(\alpha \cdot \theta_f) o_f$, where $o_f$ is the left neutral element
of the right quasigroup $(M, \Xi_f)$ and $\theta_f = \Lambda_f (o_f, o_f)^{-1}$.

In $(F5)$, we have
$$
(a^\beta)_f = \Phi_f (a, \beta) = f(\Phi(f^{-1}(a), \beta)) = f(f^{-1}(a)^ \beta),
$$
$$
({^a}\beta)_f = \Psi_f (a, \beta) = \Psi(f^{-1}(a), \beta) = {^{f^{-1}(a)}}\beta,
$$
$$
[(a^\beta)_f, b]_f = \Xi_f(\Phi_f (a, \alpha), b) = f([f^{-1}(f(f^{-1}(a)^ \alpha)), f^{-1}(b)]) =
$$
$$
= f([f^{-1}(a)^ \alpha, f^{-1}(b)]),
$$
$$
((a^\beta)_f, b)_f = \Lambda_f(\Phi_f (a, \alpha), b) = (f^{-1}(f(f^{-1}(a)^ \alpha)), f^{-1}(b)) =
$$
$$
= (f^{-1}(a)^ \alpha, f^{-1}(b)).
$$
Therefore
$$
(\alpha a \cdot \beta b)_f = (\alpha \cdot {^{f^{-1}(a)}}\beta \cdot (f^{-1}(a)^ \beta, f^{-1}(b))) \,
f([f^{-1}(a)^ \beta, f^{-1}(b)]) =
$$
$$
= \alpha f^{-1}(a) \cdot \beta f^{-1}(b).
$$
Consequently,  by assigning to each element $\alpha a \in {\overline G}$
an element $\alpha f^{-1}(a)$ we get a group homomorphism
${\overline G}_f \rightarrow {\overline G}$, which we denote
$\overline{f^{-1}}$.

Since
$$
\overline{(f \cdot g)^{-1}} = \overline{f^{-1}} \cdot \overline{g^{-1}}, \quad
\overline{1_M} = 1_{\overline G},
$$
then $\overline{f^{-1}}$ is an invertible homomorphism, i.e., it is an isomorphism.

Now. for any $\alpha \in H$, we have
$$
((f_0)_f \cdot \overline{f^{-1}})(\alpha) = \overline{f^{-1}}((f_0)_f(\alpha)) =
\overline{f^{-1}}((\alpha \cdot \theta_f) o_f) = \\
$$
$$
= \overline{f^{-1}}((\alpha \cdot \theta) f(o)) = (\alpha \cdot \theta) f^{-1}(f(o)) = f_0 (\alpha).
$$
Here we use the simple fact that $o_f = f(o)$ and, consequently,
$$
\theta_f = (0_f, 0_f)_f^{-1} = (f^{-1}(f(o)), f^{-1}(f(o)))^{-1} = (o, o)^{-1} = \theta.
$$

Thus, we have obtained the following result.

\begin{lemma}
\label{lm12}
Let $M_H$ be an arbitrary hypergroup over the group, $f \in S_M$.
Then the group extensions
\begin{center}
${\overline \varphi} = f_0: H \rightarrow {\overline G}$ and
${\overline \varphi}_f = (f_0)_f: H \rightarrow {\overline G}_f$
\end{center}
 associated with hypergroups over the group $M_H$ and $(M_H)_f$
are isomorphic, namely, ${\overline \varphi} = {\overline \varphi}_f \cdot \overline{f^{-1}}$.
\end{lemma}

This completes the proof of Theorem $\ref{thrm9}$.
\end{proof}

\begin{theorem}
\label{th10}
The mappings $\mathcal H$ and $\mathcal E$ are mutually inverse.
\end{theorem}

\begin{proof}
Let $M_H$ be a hypergroup over the group,
${\overline \varphi} = f_0: H \rightarrow {\overline G} = H \odot M$
be the corresponding monomorphism of groups.
Then the mapping $\mathcal E$ sends the orbit of $M_H$
with respect of the action $\mathcal A$
to the class of extensions isomorphic to $\overline \varphi$.

Let us consider the image $\overline H$ of $f_0$ and
 the coresrtiction ${\overline f}_0: H \rightarrow {\overline H}$.
 Similarly, let  $\overline M$ be the image of the injection
 $f_1: M \rightarrow {\overline G}$ and
 ${\overline f}_1: M \rightarrow {\overline M}$ be
 the corresponding corestriction.
 Then $\overline M$ is a complementary set to the subgroup
 $\overline H$ of the group $\overline G$
 and we have an isomorphism of hypergroups over the group:
 $$
 {\overline f} = ({\overline f}_0, {\overline f}_1): M_H \rightarrow {\overline M}_{\overline H}.
 $$
 Therefore the image of a class of extensions isomorphic to $\overline \varphi$
 with respect to the mapping $\mathcal H$ is the orbit of the hypergroup $M_H$
 under the action $\mathcal A$ of the group $\mathcal G$.

  Thus, we proved that the  composite ${\mathcal E} \cdot {\mathcal H}$
 coincides with the identical mapping of the set $Hg(M, H) \slash {\mathcal A}$.
 Now we will prove that the  composite ${\mathcal H} \cdot {\mathcal E}$
 coincides with the identical mapping of the set $Ext_m(H)$.

 Let $\varphi: H \rightarrow G$ be an arbitrary extension of the group $H$,
 $H' = \varphi (H)$, $M'$ be a complementary set
 to the subgroup $H'$ of the group $G$
 and $M'_{H'}$ be the associated hypergroup over the group,
 $\Omega'_\varphi = (\Phi'_\varphi, \Psi'_\varphi, \Xi'_\varphi, \Lambda'_\varphi)$
 be its system of structural mappings.

 Let $\phi':M \rightarrow M'$ be an arbitrary bijection,
 $\phi: H \rightarrow H'$ be the corestriction of $\varphi$,
 i.e., $\phi (\alpha) = \varphi (\alpha)$ for any $\alpha \in H$.
 Set ${\tilde \phi} = (\phi, \phi')$ and let
 $(M_H)_{\tilde \phi}$ be a hypergroup with a basic set $M$ over the group $H$,
 having a system of structural mappings
 $\Omega_{\tilde \phi} = (\Phi_{\tilde \phi}, \Psi_{\tilde \phi}, \Xi_{\tilde \phi}, \Lambda_{\tilde \phi})$,
 uniquely defined by the conditions
 $$
 \Phi_{\tilde \phi} \cdot \phi' = (\phi', \phi) \cdot \Phi'_\varphi, \quad
 \Psi_{\tilde \phi} \cdot \phi = (\phi', \phi) \cdot \Psi'_\varphi,
 $$
 $$
 \Xi_{\tilde \phi} \cdot \phi' = (\phi', \phi') \cdot \Xi'_\varphi, \quad
 \Lambda_{\tilde \phi} \cdot \phi = (\phi', \phi') \cdot \Lambda'_\varphi.
 $$
 Then, the orbit of  $(M_H)_{\tilde \phi}$ under the action $\mathcal A$
 is the image of the isomorphic class of the extension $\varphi$
 with respect to the mapping $\mathcal H$.

 The image of the orbit of  $(M_H)_{\tilde \phi}$ under the action
 $\mathcal A$  with respect to the mapping $\mathcal E$ is the
 isomorphic class of the group extension
 $f_0: H \rightarrow {\overline G}_{\tilde \phi}$,
 where the outer exact product
 ${\overline G}_{\tilde \phi} = (H \odot M)_{\tilde \phi}$
 is a group with a basic set $H \times M$ and with the binary operation
 $$
 (\alpha a \cdot \beta b)_{\tilde \phi} = (\alpha \cdot \Psi_{\tilde \phi}(a, \beta) \cdot
 \Lambda_{\tilde \phi}(\Phi_{\tilde \phi}(a, \beta), b)) \,
 \Xi_{\tilde \phi}(\Phi_{\tilde \phi}(a, \beta), b).
 $$
 The monomorphism $f_0$ is defined by the formula
 $$
 f_0 (\alpha) = (\alpha \cdot \theta_{\tilde \phi}) o_{\tilde \phi},
 $$
 where the neutral element $o_{\tilde \phi}$ of the right quasigroup
 $(M, \Xi_{\tilde \phi})$ is equal to the preimage of the neutral element
 $o'_\varphi$ of the right quasigroup $(M', \Xi'_\varphi)$ with respect
 to the map $\phi'$, and
 $\theta_{\tilde \phi} = \Lambda_{\tilde \phi}(o_{\tilde \phi}, o_{\tilde \phi})^{-1}$.

\begin{lemma}
\label{lm13}
Consider the map of sets
$$
\pi: {\overline G}_{\tilde \phi} = (H \odot M)_{\tilde \phi} \rightarrow G, \quad
\alpha a \mapsto \phi (\alpha) \cdot \phi'(a).
$$
This map determines an isomorphism of groups and $f_0 \cdot \pi = \varphi$.
\end{lemma}

\begin{proof}
First, $\pi$ is a homomorphism since
$$
\pi (\alpha a \cdot \beta b)_{\tilde \phi} =
\pi ((\alpha \cdot \Psi_{\tilde \phi}(a, \beta) \cdot
 \Lambda_{\tilde \phi}(\Phi_{\tilde \phi}(a, \beta), b)) \,
 \Xi_{\tilde \phi}(\Phi_{\tilde \phi}(a, \beta), b)) = \\
 $$
 $$
 = \phi (\alpha \cdot \Psi_{\tilde \phi}(a, \beta) \cdot
 \Lambda_{\tilde \phi}(\Phi_{\tilde \phi}(a, \beta), b)) \cdot
 \phi'(\Xi_{\tilde \phi}(\Phi_{\tilde \phi}(a, \beta), b)) = \\
$$
$$
= \phi (\alpha) \cdot \phi (\Psi_{\tilde \phi}(a, \beta)) \cdot
\phi (\Lambda_{\tilde \phi}(\Phi_{\tilde \phi}(a, \beta), b))) \cdot
 \phi'(\Xi_{\tilde \phi}(\Phi_{\tilde \phi}(a, \beta), b)) =
$$
$$
= \phi (\alpha) \cdot \Psi'_\varphi (\phi' (a), \phi (\beta)) \cdot
\Lambda'_\varphi (\Phi'_\varphi(\phi' (a), \phi (\beta)), \phi'(b)) \cdot
$$
$$
\Xi'_\varphi(\Phi'_\varphi (\phi' (a), \phi (\beta)), \phi'(b)) =
\phi (\alpha) \cdot ({^{\phi'(a)}}\phi (\beta))_\varphi \cdot
(\phi'(a)^{\phi (\beta)})'_\varphi \cdot \phi' (b) =
$$
$$
= \phi (\alpha) \cdot \phi' (a) \cdot \phi (\beta) \cdot \phi' (b) =
\pi (\alpha a) \cdot \pi (\beta b).
$$

Moreover, $\pi$ has an inverse map $\pi': G \rightarrow {\overline G}_{\tilde \varphi}$.
Indeed, for any element $x \in G$ there are unique elements $\alpha' \in H' = \varphi (H)$ and
$a' \in M'$ such that $x = \alpha' \cdot a'$. Set $\pi'(x) = (\phi^{-1}(\alpha'), (\phi')^{-1}(a'))$.
It is not difficult to check that $\pi' = \pi^{-1}$.
Consequently, $\pi$ is an isomorphism.

Second, for any $\alpha \in H$ we have
$$
(f_0 \cdot \pi) (\alpha) = \pi (f_0 (\alpha)) =
\pi ((\alpha \cdot \theta_{\tilde \phi}) o_{\tilde \phi}) =
\phi (\alpha \cdot \theta_{\tilde \phi}) \cdot \phi'(o_{\tilde \phi}) = \\
$$
$$
= \phi (\alpha) \cdot \phi (\theta_{\tilde \phi}) \cdot \phi'(o_{\tilde \phi})
= \varphi (\alpha) \cdot \theta'_\varphi \cdot o'_\varphi = \varphi(\alpha).
$$
Here we use that
$$
\phi (\theta_{\tilde \phi}) = \phi (\Lambda_{\tilde \phi}(o_{\tilde \phi}, o_{\tilde \phi})) =
\Lambda'_\varphi(\phi'(o_{\tilde \phi}), \phi'(o_{\tilde \phi})) =
\Lambda'_\varphi(o'_\varphi, o'_\varphi)) = \theta'_\varphi.
$$
The lemma is proved.
\end{proof}

Thus, we proved that
the composite of mappings ${\mathcal H} \cdot {\mathcal E}$ is equal to the identical mapping
of the set $Ext_m (H)$,
and, consequently, we have proved the main Theorem $\ref{th10}$.
\end{proof}

 In the last two sections,  by imposing various restrictions on hypergroups over the group
and  groups extensions, we obtain some particular cases of the main theorem,
including Schreier's  theorem of 1926 (see $\cite{S1}, \cite{S2}$ and also $\cite{K}, \cite{R}$).


\section{Hypergroups over the group with a trivial action $\Phi$
and normal extensions of groups}

Let $M_H$ be a hypergroup over the group
with a system of structural mappings
$\Omega = (\Phi, \Psi, \Lambda, \Xi)$ and
$G = H \odot M$ be the associated
(outer) exact product.

\begin{definition}
We say that $M_H$ is a hypergroup with a trivial
action $\Phi$ of the group $H$ on $M$, if
$a^\alpha = a$ for any $a \in M$ and $\alpha \in H$.
Denote by $Hg_0(M, H)$ the subset of $Hg(M, H)$
consisting of all hypergroups over the group with a
trivial action $\Phi $ of the group $H$ on $M$.
\end{definition}

\begin{proposition}
\label{prop6}
Let $M_H$ be a hypergroup over the group
with a trivial action $\Phi$ of the group $H$ on $M$.
Then, we can assume that for such a
hypergroup over the group
the system of structural mappings is
$\Omega = (\Psi, \Xi, \Lambda)$,
and conditions $P1 - P4$
(of the definition of a hypergroup over the group)
take the following form.

$(A)$ The structural mapping $\Xi$
determines on $M$ a structure
of a group with a neutral element $o$
and with an inverse to $a \in M$ element $a^{[-1]}$.

$(B)$ The structural mapping $\Psi$
determines a mapping
$$
{\underline \Psi}: M \rightarrow Aut H, \quad
a \rightarrow {\underline \Psi}_a, \quad
{\underline \Psi}_a (\alpha) = \Psi (a, \alpha) = {^a}\alpha,
$$
which satisfies the condition
$$
{^{[a, b]}}\alpha  = (a, b)^{-1} \cdot {^a}({^b}\alpha) \cdot (a, b).
$$

$(C)$ The structural mapping $\Lambda$
satisfies the condition
$$
(a, b) \cdot ([a, b], c) = {^a}(b, \, c) \cdot (a, [b, c]).
$$
\end{proposition}

\begin{proof}
First of all, condition $P2$ and
identity $(A2)$ of condition $P4$
together with structural mapping $\Phi$
become trivial and disappear.

Identity $(A4)$ of conditions $P4$ is
reduced to the associative law.
This law together with condition $P1$ is
equivalent to condition $(A)$,
since a group is exactly an associative
right quasigroup with a left neutral element.

Identity $(A5)$ is reduced to the identity
of condition $(C)$.

Now for an arbitrary $a \in M$, let us consider the map
$$
{\underline \Psi}_a: H \rightarrow H, \quad
{\underline \Psi}_a (\alpha) = \Psi (a, \alpha) = {^a}\alpha.
$$
Identity $(A1)$ is reduced to the condition
that any map of this form is a  homomorphism.
In Lemma $\ref{lm14}$ below  we prove
that any ${\underline \Psi}_a$ is invertible and,
consequently, is an automorphism.
Finally, identity $(A3)$ is reduced
to the identity of condition $(B)$.

Note that this condition contains not
only identities $(A1)$ and $(A3)$,
but also condition $P3$, since
any automorphism is a surjective map.
Proposition $\ref{prop6}$ is proved.
\end{proof}

\begin{lemma}
\label{lm14}
Let $M_H$ be a hypergroup over the group
with a trivial action $\Phi$ of the group $H$ on $M$.
Then, for any $a \in M$, ${\underline \Psi}_a$
is an invertible map.
\end{lemma}

\begin{proof}
The identity of condition (B) has the functional form
$$
\hspace{30mm}
{\underline \Psi}_{[a. b]} = {\underline \Psi}_b \cdot {\underline \Psi}_a \cdot C_{(a, b)},
\hspace{25mm} (F6)
$$
where $C_\gamma$ means the conjugation
by element $\gamma$:
$$
C_\gamma (\alpha) = \gamma^{-1} \cdot \alpha \cdot \gamma.
$$

Since, according to $(A)$, $(M, \Xi)$ is a group,
 we have
$$
[a^{[-1]}, a] = o = [a, a^{[-1]}].
$$
Substituting  $a = b^{[-1]}$  in $(F5)$,  we get
$$
C_\theta = {\underline \Psi}_b \cdot {\underline \Psi}_{b^{[-1]}}\cdot C_{(b^{[-1]}, b)}.
$$
Hence, ${\underline \Psi}_b$ has a right inverse map
$$
{\underline \Psi}_b^{-1} =
{\underline \Psi_{b^{[-1]}}} \cdot C_{(b^{[-1]}, b) \cdot \theta^{-1}}.
$$
Similarly, substituting  $b = a^{[-1]}$  in $(F6)$,
we get that there exists also  a left inverse map.
\end{proof}

\begin{proposition}
\label{prop7}
For any hypergroup over the group $M_H$
with a trivial action $\Phi$ of the group $H$ on $M$
and for any element $f \kappa$ of $\mathcal G$,
the hypergroup over the group $(M_H)_{f \kappa}$
is with trivial action $\Phi_{f \kappa}$
of the group $H$ on $M$.

In the case of hypergroups over the group
with a trivial action $\Phi$ of the group $H$ on $M$,
formulas $(i)_f -m(iv)_f$ and $(i)_\kappa - (iv)_\kappa$
are as follows.
Relations $(i)_f$ and $(i)_\kappa$ become tautologies.
Relations $(ii)_f - (iv)_f$ preserve their form.
Formulas $(ii)_\kappa - (iv)_\kappa$ take
the following form:
\begin{enumerate}
\item[$\rm (ii)_{\kappa. 0}$] \quad
$({^ a}\alpha)_\kappa =
\kappa_a \cdot {^a}\alpha \cdot \kappa_a^{-1}$,
\item[$\rm (iii)_{\kappa, 0}$] \quad
$[a, \, b]_\kappa = [a, \, b]$,
\item[$\rm (iv)_{\kappa, 0}$] \quad
$(a, b)_\kappa =
\kappa_a \cdot {^a}\kappa_b \cdot (a, \, b) \cdot
\kappa_{[a,\, b]}^{-1}$.
\end{enumerate}
\end{proposition}

\begin{proof}
If the action $\Phi$ is trivial, then for any
$a \in M, \, \alpha \in H$

($\kappa$) according to relation
$(i)_\kappa$ of Section 4
$$
\Phi_\kappa (a, \alpha) = \Phi (a, \alpha) = a;
$$

($f$) $f^{-1}(a)^\alpha = f^{-1}(a)$,
consequently, $f(\Phi(f^{-1}(a), \alpha)) = a$,
which means that $\Phi_f (a, \alpha) = a$.

Since $(M_H)_{f \kappa} = ((M_H)_f)_\kappa$,
the first part of Proposition $\ref{prop7}$ is proved.
The second part of this proposition is evident.
\end{proof}

\begin{corollary}
(a) An orbit from $Hg(H, M) \slash {\mathcal A}$
contains a hypergroup over the group
with a trivial action $\Phi$ of the group $H$ on $M$
if and only if all elements of this orbit
are hypergroups over the group
with a trivial action $\Phi$ of the group $H$ on $M$.

(b) The subset $Hg_0(M, H)$ of the set $Hg(M, H)$
is invariant under action $\mathcal A$
of the group $\mathcal G$.
Therefore, by restricting $\mathcal A$
on $Hg_0(M, H)$, we get an action
of $\mathcal G$ on  $Hg_0(M, H)$.
There is a natural embedding
$$
Hg_0(H. M) \slash {\mathcal A} \subset Hg(H. M) \slash {\mathcal A}.
$$
\end{corollary}

A group extension $\varphi: H \rightarrow G$
is called $normal$, if
$H' = \varphi (H)$ is a normal subgroup of $G$.
Clearly, an extension, isomorphic to
a normal extension, is normal as well.
Therefore, a class of isomorphic extensions
contain a normal extension
if and only if all extensions of this class
are normal.
Let us denote by $Ext_{m, 0}(H)$ a subset of $Ext_m (H)$,
consisting of all classes of isomorphic normal extensions.

\begin{theorem}
\label{th11}
Let a hypergroup $M_H$ and a group extension $\varphi$
correspond to one another under mappings $\mathcal E$
and $\mathcal H$.
Then, $\varphi$ is a normal extension if and
only if $M_H$ is a hypergroup over the group
with a trivial action $\Phi$ of the group $H$ on $M$.
Therefore, the restrictions of the mappings
$\mathcal E$ and $\mathcal H$
on $Hg_0(M, H)$ and $Ext_{m, 0}(H)$,
respectively, determine  a bijection.
\end{theorem}

\begin{proof}
The triviality of the structural mapping
$\Phi$ means that
$a \cdot \alpha = {^a}\alpha \cdot a$
for any $a \in M$ and $\alpha \in H$.
It means that the corresponding extension
$\varphi$ is normal.
The second part of Theorem $\ref{th11}$
is an immediate consequence
of Theorem $\ref{th10}$.
\end{proof}

Now, let $M_H, (M_H)' \in Hg_0(M, H)$ and
$\Omega = (\Psi, \Xi, \Lambda)$,
$\Omega' = (\Psi', \Xi', \Lambda)$
be their systems of structural mappings.
Then, there is an element $f \kappa \in {\mathcal G}$
such that $(M_H)' = (M_H)_{f \kappa}$
and, according to $\rm{(iii)}_{\kappa, 0}$
of Proposition 7
and $\rm{(iii)}_f$ of Section 4, we have
$\Xi' = \Xi_{f \kappa} = \Xi_f$,
that is the groups $(M, \Xi)$ and $(M, \Xi')$
are isomorphic.

Let $Hg_0 (H, M, [\Xi])$ be a subset of $Hg_0(H, M)$, consisting
of all hypergroups $(M_H)'$, having a structural mapping $\Xi'$,
isomorphic to a fixed binary (group) operation $\Xi$. This means that
$\Xi' = \Xi_f$ for some $f \in S_M$.
Then, $Hg_0 (H, M, [\Xi])$ is invariant under the action of $\mathcal G$,
the group $\mathcal G$ acts on $Hg_0 (H, M, [\Xi])$ and
$$
Hg_0 (H, M, [\Xi]) \slash {\mathcal A} \subset Hg_0 (H, M) \slash {\mathcal A}.
$$
Let us emphasize that for any $(M_H)' \in Hg_0 (H, M, [\Xi])$ we have that
$(M, \Xi')$ is a group, isomorphic to the group $(M, \Xi)$.

\begin{definition}
Let $H$ and $Q$ be arbitrary groups, $\varphi: H \rightarrow G$ be
a normal extension such that the quotient-group $G \slash \varphi (H)$
is isomorphic to $Q$. In this case, we say that
$G$ is an {\it extension of $H$ by $Q$} or, equivalently,
$G$ is a {\it coextension of $Q$ by $H$}.
\end{definition}

If a normal extension $\varphi: H \rightarrow G$ has the
quotient-group $G \slash \varphi (H)$, which is  isomorphic to $Q$,
obviously, any
normal extension $\varphi': H \rightarrow G'$, isomorphic to $\varphi$,
has the quotient-group $G \slash \varphi' (H)$, which is isomorphic to $Q$.
Therefore, we have a correctly defined subset $Ext(H, Q)$ of
$Ext_m(H)$, where $m = |Q|$, consisting of isomorphic classes of all
normal extensions $\varphi$ with the quotient-group isomorphic to $Q$,
that is consisting of isomorphic classes of all extensions of $H$ by $Q$.

\begin{theorem}
\label{th12}
By restriction of mappings $\mathcal E$ and $\mathcal H$
a bijection is established between the sets
$Hg_0(H, M, [\Xi]) \slash {\mathcal A}$ and
$Ext(H, Q)$, where $Q = (M, \Xi)$.
\end{theorem}

\begin{proof}
Let  $(M_H)'$ be a representative of an orbit from
$Hg_0(H, M, [\Xi]) \slash {\mathcal A}$.
Then, the image of this orbit with respect
to $\mathcal E$
is the isomorphic class of the extension
$$
\varphi: H \rightarrow G = H \odot M, \quad
\alpha \mapsto (\alpha \cdot \theta)o.
$$
Since the considered orbit is from
$Hg_0(H, M, [\Xi]) \slash {\mathcal A}$,
the extension $\varphi$ is normal and the quotient-group
$G \slash {\varphi}(H)$ is isomorphic to $Q = (M, \Xi)$.
Consequently, the isomorphic class of $\varphi$ belongs
to $Ext(H, Q)$.

Conversely, suppose that the extension $\varphi: H \rightarrow G$
represents an element from $Ext (H, Q)$. Consider a transversal $M$
to the subgroup $\varphi (H)$ of $G$ and corresponding
hypergroup $M_{\varphi (H)}$ over the group. Then, the structural
mapping $\Xi$ of $M_{\varphi (H)}$ is isomorphic to the binary
operation of the group $Q$ and is isomorphic to the corresponding
structural mapping of ${\mathcal H} ({\varphi})$.
Therefore,
${\mathcal H} ({\varphi}) \in Hg_0 (H, M, [\Xi]) \slash {\mathcal A}$.
\end{proof}

In the sequel, we need a modification of this theorem, where
the set $Hg_0(H, M, [\Xi])$ is replaced by its subset
$Hg_0 (H, M, \Xi)$, consisting of all hypergroups
$M_H$, having the same structural
  mapping $\Xi$.

The following proposition immediately follows from the definition
of the group $\mathcal G$ (see Section 4,  Proposition 5).
\

\begin{proposition}
\label{pr8}
The subset $H^M \times \{ 1_M \}$
of the group $\mathcal G$
forms a subgroup of $\mathcal G$, which is
reverse isomorphic to the group $H^M$.
\end{proposition}

We denote the obtained subgroup of $\mathcal G$
by $(H^M)_{rev}$.

\begin{lemma}
\label{lm15}
The subset $Hg_0(H, M, \Xi)$ of $Hg(H, M)$
is closed under the action of the group $(H^M)_{rev}$.
By restricting $\mathcal A$ an action
$$
{\mathcal A}_o: (H^M)_{rev} \times Hg_0(H M, \Xi)
\rightarrow Hg_0(H, (M, \Xi)), \quad M_H \mapsto
(M_H)_\kappa
$$
of group $(H^M)_{rev}$ on $Hg_0(H M, \Xi)$
is determined.
\end{lemma}

\begin{proof}
It directly follows from the item
${\rm (iii)}_\kappa$
of Section 4 and Proposition $\ref{pr8}$ above.
\end{proof}

\begin{proposition}
\label{prop 9}
There exists a canonical bijection
$$
\iota: Hg_0(H, M, \Xi) \slash (H^M)_{rev}
\rightarrow Hg_0(H, M, [\Xi]) \slash {\mathcal G},
$$
which sends the orbit of $M_H$ under the action
${\mathcal A}_0$ to the orbit of $M_H$ under the action
$\mathcal A$.
\end{proposition}

\begin{proof}
The mapping $\iota$ is correctly defined and is
injective and surjective according to definitions
of the actions $\mathcal A$ and ${\mathcal A}_0$.
\end{proof}

Using this proposition, we get a modification of Theorem 12.

\begin{theorem}
\label{th13}
The composite $\iota \cdot {\mathcal E}$
determines a  bijection between the sets
$Hg_0(H, M, \Xi) \slash (H^M)_{rev}$ and
$Ext(H, Q)$, where $Q = (M, \Xi)$.
\end{theorem}


\section{Commutative case and Schreier's theorem}

In this section, we suppose that the group $H$ is commutative.
Then, the formula of item (B) of Proposition 6 is reduced
to the form
$$
(B_{com}) \hspace{20mm} \Psi ([a, b], \alpha) = \Psi(a, \Psi(b, \alpha)).
\hspace{25mm}
$$
This means that $\Psi$ is a reverse action
of the group $(M, \Xi)$ on the group $H$ by automorphisms.
Thus, the condition $(B_{com})$
is equivalent to condition that
the mapping
$$
({\underline B}_{com}) \hspace{15mm}
{\underline \Psi}: (M, \Xi) \rightarrow Aut \, H, \quad
{\underline \Psi}_a(\alpha) = \Psi (a, \alpha) = {^a}\alpha
\hspace{20mm}
$$
is a reverse representation
(= reverse homomorphism of groups).

Condition (C) of Proposition 6 preserves its form
$$
(C_{com}) \hspace{20mm}(a, b) \cdot ([a, b], c) =
{^a}(b, \, c) \cdot (a, [b, c]),
\hspace{25mm}
$$
but now the operation of multiplication is commutative.

Formula $(iii)_{\kappa, 0}$ of Section 6
preserves its form as well, and formulas
$(ii)_{\kappa, 0}$  and $(iv)_{\kappa, 0}$
take the form
$$
(ii)_{\kappa, 0, com} \quad \quad ({^ a}\alpha)_\kappa = {^a}\alpha,
\quad \quad
$$
$$
(iv)_{\kappa, 0, com} \quad \quad
(a, b)_\kappa = (a, \, b) \cdot
\kappa_a \cdot {^a}\kappa_b \cdot \kappa_{[a,\, b]}^{-1}.
\quad \quad
$$

Let $Hg_0(H_{com}, M, \Xi, \Psi)$ be a subset of the set
$Hg_0(H_{com}, M, \Xi)$, consisting of all hypergroups
$M_H \in Hg_0(H_{com}, M, \Xi)$ with a fixed structural
mapping $\Psi$, and $Ext (H, (M, \Xi), {\underline \Psi})$
be a subset of the set $Ext (H, (M, \Xi))$, consisting of
all extensions from $Ext (H, (M, \Xi))$ with a fixed reverse
representation $\underline \Psi$.

\begin{theorem}
\label{th14}
The subset $Hg_0(H_{com}, M, \Xi, \Psi)$ of
$Hg_0(H_{com}, M, \Xi)$ is invariant under the action
${\mathcal A}_0$, consequently, we get a subset
$$
Hg_0(H_{com}, M, \Xi, \Psi) \slash {\mathcal A}_0
$$
of the set $Hg_0(H_{com}, M, \Xi) \slash {\mathcal A}_0$.
Hence, the restrictions of mappings ${\mathcal E}$ and
${\mathcal H}$, respectively, to
$Hg_0(H_{com}, M, \Xi, \Psi) \slash {\mathcal A}_0$
and $Ext(H_{com}, (M. \Xi), {\underline \Psi})$
determine mutually inverse mappings and,
 consequently, a bijection of corresponding sets.
\end{theorem}

\begin{proof}
The first part of this theorem follows from
$(ii)_{\kappa, 0, com}$, while
the second part follows from Theorem 13.
\end{proof}

The set
$Hg_0(H_{com}, M, \Xi, \Psi) \slash {\mathcal A}_0$
has another well known interpretation.
This interpretation
is based on the following four propositions.

\begin{proposition}
\label{prop10}
Let $H$ be a multiplicative group, $M$ be an arbitrary set,
$H^{M \times M}$ be the set of all mappings from the Cartesian
product $M \times M$ to $H$. Define a binary operation $\ast$ on
$H^{M \times M}$ by the formula
$$
(\Lambda \ast \Lambda') (a,b) =
\Lambda (a, b) \cdot \Lambda' (a, b),
$$
where $\Lambda, \Lambda' \in H^{M \times M}$ and $a, b \in M$.
This operation is associative, has a neutral element
\begin{center}
$O(a, b) = \varepsilon$ for all $a, b \in M$,
\end{center}
where $\varepsilon$ is the neutral element of $H$,
and for any element $\Lambda$ of $H^{M \times M}$ there is
an inverse element $\Lambda^{-1}$, determined by
\begin{center}
$\Lambda^{-1} (a, b) = \Lambda (a, b)^{-1}$
for any $a, b \in  M$.
\end{center}
Thus, the set $H^{M \times M}$ together with the
binary operation $\ast$ is a group.
\end{proposition}

\begin{proof}
The verification of the group axioms is trivial.
\end{proof}

The group $(H^{M \times M}, \ast)$ is denoted by $C^2(M, H)$
and is called a {\it group of two-dimensional cochains}.
Its elements are called {\it two-dimensional cochains}.

For  hypergroups over the group
$M_H \in Hg_0(H_{com}, M, \Xi, \Psi)$
only the structural mapping $\Lambda$
is changed, the other structural mappings are fixed.
So, we can identify
$Hg_0(H_{com}, M, \Xi, \Psi)$ with a
subset of the set $H^{M \times M}$, when
the group $H$ is commutative and $M$ is
a group with the binary operation $\Xi$.
Let us emphasize, that
$Hg_0(H_{com}, M, \Xi, \Psi)$ coincides with
the subset of the set $H^{M \times M}$,
defined by condition $(C_{com})$.

\begin{proposition}
\label{prop11}
The subset $Hg_0(H_{com}, M, \Xi, \Psi)$ of the
set $H^{M \times M}$ is closed under the binary operation
$\ast$ and forms a subgroup of the group
$(H^{M \times M}, \ast)$.
\end{proposition}

\begin{proof}
Let us check that if $\Lambda, \Lambda' \in H^{M \times M}$
satisfy the condition $(C_{com})$:
$$
\Lambda(a, b) \cdot \Lambda([a, b], c) =
{^a}\Lambda(b, \, c) \cdot \Lambda(a, [b, c]),
$$
$$
\Lambda'(a, b) \cdot \Lambda'([a, b], c) =
{^a}\Lambda'(b, \, c) \cdot \Lambda'(a, [b, c]),
$$
then $\Lambda \ast \Lambda'$ and $\Lambda^{-1}$
also satisfy this condition. Indeed
$$
(\Lambda \ast \Lambda')(a, b) \cdot
(\Lambda \ast \Lambda')([a, b], c) =
\Lambda (a, b) \cdot \Lambda'(a, b) \cdot
\Lambda ([a, b], c) \cdot \Lambda'([a, b], c) =
$$
$$
= \Lambda (a, b) \cdot \Lambda ([a, b], c) \cdot
\Lambda'(a, b) \cdot \Lambda'([a, b], c) =
$$
$$
= {^a}\Lambda (b, \, c) \cdot \Lambda (a, [b, c]) \cdot
{^a}\Lambda'(b, \, c) \cdot \Lambda'(a, [b, c]) =
$$
$$
= {^a}\Lambda (b, \, c) \cdot {^a}\Lambda'(b, \, c) \cdot
\Lambda (a, [b, c]) \cdot\Lambda'(a, [b, c]) =
{^a}(\Lambda \ast \Lambda')(b, \, c) \cdot
(\Lambda \ast\Lambda')(a, [b, c]),
$$
$$
\Lambda^{-1}(a, b) \cdot \Lambda^{-1}([a, b], c) =
\Lambda (a, b)^{-1} \cdot \Lambda([a, b], c)^{-1} =
(\Lambda (a, b) \cdot \Lambda([a, b], c))^{-1} =
$$
$$
= ({^a}\Lambda (b, \, c) \cdot \Lambda (a, [b, c]))^{-1} =
{^a}\Lambda (b, \, c)^{-1} \cdot \Lambda (a, [b, c])^{-1} =
$$
$$
= {^a}\Lambda^{-1}(b, \, c) \cdot \Lambda^{-1}(a, [b, c]).
$$
The proposition is proved.
\end{proof}

The group $Hg_0(H_{com}, M, \Xi, \Psi)$
is denoted by $Z^2((M, \Xi), H, \underline{\Psi})$ and
is called a {\it group of two-dimensional cocycles}.
The elements of this group
are called {\it systems of factors} $\cite{K}$ or
{\it factor sets} $\cite{R}$ or {\it two-dimensional
cocycles}.

For an arbitrary $\kappa \in H^M$,
an element $\Lambda_\kappa$ of $H^{M \times M}$
is defined as follows:
$$
\Lambda_\kappa (a, b) = \kappa_a \cdot
{^a}\kappa_b \cdot \kappa_{[a, b]}^{-1},
$$
where as usual $[a, b] = \Xi (a, b)$.
Recall that by definition
$$
\kappa_a^{-1} = (\kappa^{-1})_a = (\kappa_a)^{-1}.
$$

\begin{proposition}
\label{prop12}
For any $\kappa, \lambda \in H^M$
the following formulas are true:
$$
\Lambda_\kappa \ast \Lambda_\lambda =
\Lambda_{\kappa \ast \lambda}, \quad
\Lambda_\kappa^{-1} = \Lambda_{\kappa^{-1}}.
$$
Consequently,
$\{ \Lambda_\kappa, \, \kappa \in H^M  \}$
forms a subgroup of the group $(H^{M \times M}, \ast)$.
\end{proposition}

\begin{proof}
For any $a, b \in M$ we have
$$
(\Lambda_\kappa \ast \Lambda_\lambda)(a, b) =
\Lambda_\kappa(a, b) \cdot \Lambda_\lambda (a, b) =
(\kappa_a \cdot {^a}\kappa_b \cdot \kappa_{[a, b]}) \cdot
(\lambda_a \cdot {^a}\lambda_b \cdot \lambda_{[a, b]}) =
$$
$$
= (\kappa_a \cdot \lambda_a) \cdot
({^a}\kappa_b \cdot {^a}\lambda_b) \cdot
(\kappa_{[a, b]}) \cdot \lambda_{[a, b]}) =
(\kappa \ast \lambda)_a \cdot
{^a}(\kappa \ast \lambda)_b \cdot
(\kappa \ast \lambda)_{[a, b]} =
$$
$$
= \Lambda_{\kappa \ast \lambda}(a, b),
$$
$$
\Lambda_\kappa^{-1}(a, b) =
(\Lambda_\kappa)^{-1}(a, b) =
\Lambda_\kappa (a, b)^{-1} =
(\kappa_a \cdot {^a}\kappa_b \cdot \kappa_{[a, b]})^{-1} =
$$
$$
= \kappa_a^{-1} \cdot {^a}(\kappa_b^{-1}) \cdot \kappa_{[a, b]}^{-1}  =
\Lambda_{\kappa^{-1}}(a, b).
$$
The proposition is proved.
\end{proof}

The elements $\Lambda_\kappa, \, \kappa \in H^M$
are called {\it two-dimensional coboundaries},
the group $(\{ \Lambda_\kappa, \, \kappa \in H^M  \}, \ast)$
 is called a {\it group of two-dimensional coboundaries}
and is denoted by $B^2((M, \Xi), H, \underline{\Psi})$.

\begin{proposition}
\label{prop15}
For any $\kappa \in H^M$ the element
$\Lambda_\kappa$ of $H^{M \times M}$
satisfies the condition $(C_{com})$.
Consequently, $B^2((M, \Xi), H, \underline{\Psi})$
is a subgroup of the group
$Z^2((Mm \Xi), H, \underline{\Psi})$.
\end{proposition}

\begin{proof}
Below,  for any $x, y \in M$, we use in $(C_{com})$
the following notation:
$$
(x, y) = \Lambda_\kappa (x, y) =
\kappa_x \cdot {^x}\kappa_y \cdot \kappa_{[x, y]}^{-1}.
$$
Then, equality $(C_{com})$ is true because
$$
{^a}(b, \, c) \cdot (a, [b, c]) =
{^a}(\kappa_b \cdot {^b}\kappa_c \cdot \kappa_{[b, c]}^{-1}) \cdot
(\kappa_a \cdot {^a}\kappa_{[b, c]} \cdot \kappa_{[a, [b, c]]}^{-1}) =
$$
$$
= {^a}\kappa_b \cdot {^a}({^b}\kappa_c) \cdot {^a}\kappa_{[b, c]}^{-1} \cdot
\kappa_a \cdot {^a}\kappa_{[b, c]} \cdot \kappa_{[a, [b, c]]}^{-1} =
$$
$$
= \kappa_a \cdot {^a}\kappa_b \cdot {^a}({^b}\kappa_c) \cdot
\kappa_{[a, [b, c]]}^{-1} =
\kappa_a \cdot {^a}\kappa_b \cdot
{^{[a, b]}}\kappa_c \cdot \kappa_{[[a, b], c]}^{-1} =
$$
$$
= (\kappa_a \cdot {^a}\kappa_b \cdot \kappa_{[a, b]}^{-1})
\cdot  (\kappa_{[a, b]} \cdot
{^{[a, b]}}\kappa_c \cdot \kappa_{[[a, b], c]}^{-1}) =
(a, b) \cdot ([a, b], c).
$$
The proposition is proved.
\end{proof}

When $H$ is an abelian group, the group
$Z^2((M, \Xi), H, \underline{\Psi})$
is commutative. Consequently,
$B^2((M, \Xi), H, \underline{\Psi})$ is
its normal subgroup.
Therefore we can consider
the quotient-group
$Z^2((M, \Xi), H, \underline{\Psi}) \slash
B^2((M, \Xi), H, \underline{\Psi})$.
This quotient-group is denoted by
$H^2((M, \Xi), H, \underline{\Psi})$ and
is called a {\it two-dimensional cohomology group}.

\begin{theorem}
\label{th15}
Let $H$ be a commutative group. Let  $M_H$ be an arbitrary
representative of an orbit from
$Hg_0(H, M, \Xi, \Psi) \slash {\mathcal A}_0$
and $\Lambda \in H^{M \times M}$ be its structural
mapping.
Then,
by assigning to the orbit of $M_H$
the element $\Lambda \ast B^2((M, \Xi), H, \underline{\Psi})
\in H^2((M, \Xi). H, \underline{\Psi})$,
a bijection
$$
j: Hg_0(H, M, \Xi, \Psi) \slash {\mathcal A}_0
\rightarrow H^2((M, \Xi). H, \underline{\Psi})
$$
is correctly determined.
\end{theorem}

\begin{proof}
The assertion of this theorem immediately follows
from $(iv)_{\kappa, 0, com}$ and the definition
of a quotient-group.
\end{proof}
Combining Theorems 14 and 15 and denoting $Q = (M, \Xi)$,
we get the theorem of Schreier-1926 ($\cite{S1}, \cite{S2}$
see also $\cite{K}, \cite{R}$).

\begin{theorem}
\label{th16}
Let $H$ be an abelian group and $Q$ be an arbitrary group.
Then there exists a canonical bijection between
the set of extensions
$Ext (H, Q, \underline{\Psi})$ of  $H$ by $Q$
with an arbitrary fixed reverse homomorphism
$\underline{\Psi}: Q \rightarrow Aut H$
and the two-dimensional cohomology group
$H^2 (Q, H, \underline{\Psi})$.
This bijection is given by the composite
${\mathcal H} \cdot \iota \cdot j$.
\end {theorem}

\begin{remark}
Let $H$ and $Q$ be some groups,
$Ext(H, Q)$ be the set of isomorphic classes
of extensions of $H$ by $Q$.
We have proved that the set $Ext(H, Q)$ is parametrized
by the two-dimensional cohomology groups
$H^2 (Q, H, \underline{\Psi})$, where $\underline \Psi$
runs over all reverse representations $Q \rightarrow Aut H$,
under the condition that the group $H$ is abelian.
We have used this condition in three points:

(a) to rearrange the factors in the right-hand side
of the relation (A3),

(b) to have that $B^2(Q, H, \underline{\Psi})$ is
a normal subgroup of $Z^2(Q, H, \underline{\Psi})$,

(c) $B^2(Q, H, \underline{\Psi}) =
\{\Lambda_\kappa, \, \kappa \in H^M \}$,
where $Q = (M, \Xi)$,
is closed with respect to operation $\ast$.

For (a), a necessary and sufficient condition is that
the image $im \Lambda$ of the structural mapping $\Lambda$
lies in the center $Z(H)$ of the group $H$.
If $im \Lambda \subset Z(H)$, then evidently
$\Lambda^{-1} \ast \Lambda_{\kappa} \ast \Lambda = \Lambda_\kappa$
for any $\kappa \in H^M$.
But for (c), a necessary and sufficient condition is
the commutativity of the group $H$.
Therefore, a description of $Ext (H, Q, , \underline{\Psi})$
as in Theorem $\ref{th16}$ for a noncommutative group $H$
is impossible.
\end{remark}

\end{document}